\documentclass[11pt]{article}
\usepackage{geometry}
\usepackage{amssymb}
\usepackage{latexsym}
\usepackage{amsmath}
\usepackage{indentfirst}
\usepackage{graphicx}
\usepackage[colorlinks=true]{hyperref}
\usepackage{mathrsfs}
\usepackage{color}
\usepackage{extarrows}

\usepackage{tikz}

\newtheorem{Thm}{Theorem}[section]

\newtheorem{Lem}{Lemma}[section]
\newtheorem{Pro}{Proposition}[section]

\newcommand{\R}{\mathbb{R}}

\numberwithin{equation}{section} \numberwithin{figure}{section}

\newenvironment{proof}{\medskip\par\noindent{\bf Proof\/}.\quad}{\qquad
\raisebox{-0.5mm}{\rule{2.5mm}{2.5mm}}\vspace{7pt}}

\begin{document}
\title{Nonlocal problems with Hardy-Littlewood-Sobolev critical exponent and Hardy potential}

\author{\quad Guangze Gu$^{1,2}$,\quad Aleks Jevnikar$^{1}$ \\
\footnotesize{\em
1.  Department of Mathematics, Computer Science and Physics, University of Udine, Via delle }\\
\footnotesize{\em
 Scienze 206, 33100 Udine, Italy.}\\
\footnotesize{\em
2. Department  of  Mathematics, Yunnan  Normal  University, Kunming 650500, China. }\\
\footnotesize{Email: guangzegu@163.com, aleks.jevnikar@uniud.it}
}

\date{}

\date{} \maketitle
\begin{abstract}
We are concerned with a Brezis-Nirenberg type problem for a
critical Choquard equation, in the sense of Hardy-Littlewood-Sobolev inequality, and with the Hardy potential in a smooth bounded domain. By exploiting variational methods we obtain existence results, which extend to different perturbation terms. Some estimates of independent interest about a nonlocal minimization problem are also derived.
\end{abstract}

\vspace{6mm} \noindent{\bf Keywords:} Choquad equation; Hardy potential; Critical exponent; Variational method.

\vspace{6mm} \noindent

\section{Introduction}

In this paper, we consider the following nonlocal problem
\begin{equation}\label{eq1.1}
\left\{\begin{array}{ll}
-\Delta u -\mu\frac{u}{|x|^2}=
\Big( \displaystyle{\int_{\Omega}} \frac{|u(y)|^{2_{\alpha}^*} }{|x-y|^{\alpha}}dy\Big) |u|^{{2_{\alpha}^*}-1}+\lambda f(u),     \,  &x\in  \Omega, \\
u =0,    \, &  x\in \partial \Omega,
\end{array} \right.
\end{equation}
where $\Omega\subset\mathbb{R}^N(N\geq3)$ is a smooth bounded domain with smooth boundary $\partial \Omega$, $0\in \Omega$, $0<\alpha<N-(N-4)_+$, $0<\mu<\bar{\mu}=\frac{(N-2)^2}{4}$, $2_{\alpha}^*:=\frac{2N-\alpha}{N-2}$ is the upper critical exponent in the sense of the Hardy-Littlewood-Sobolev inequality. Moreover, the problem features the  so-called Hardy potential.
This potential  arises in many physical contexts such as  nonrelativistic quantum mechanics,  molecular physics and  quantum cosmology (see \cite{Felli-Marchini-Terracini2007JFA}, \cite{Frank-Land-Spector1971RMP}, \cite{Landau-Lifshitz1985Book} ). Since the Hardy potential does not belong to Kato class (\cite{Reed-Simon1978Book}), equation \eqref{eq1.1} can  be seen as doubly critical problem, which brings some new difficulties and analytical challenges, especially with respect ot the lack of compactness.

When $f(u)=\lambda u$ and $\alpha \to N$, equation \eqref{eq1.1} reduces to the following elliptic problem
\begin{equation}\label{eq1.3}
\left\{\begin{array}{ll}
-\Delta u -\mu\frac{u}{|x|^2}=
|u|^{2^*-1}+\lambda u,     \,  &x\in  \Omega, \\
u(x) =0,    \, &  x\in \partial \Omega,
\end{array} \right.
\end{equation}
where $2^*:=\frac{2N}{N-2}$ is the critical Sobolev exponent.
The existence of solution of equation \eqref{eq1.3} for $\mu<\bar{\mu}$ depends  on the spectrum $\sigma_\mu$ of the operator $-\Delta u -\frac{\mu}{|x|^2}$ with Dirichlet boundary conditions.
Jannelli \cite{Jannelli1999JDE} proved that equation \eqref{eq1.3} possesses  at least one positive solution provided that $\mu\in(0 ,\bar{\mu}-1]$ and $\lambda\in (0,\lambda_1)$ or $\Omega=B(0,1)$ is a ball, $\mu\in (\bar{\mu}-1, \bar{\mu}) $  and $\lambda\in (\lambda^*, \lambda_1)$, where $\lambda^*\in (0, \lambda_1)$  is a suitable constant depending on $\mu$ and  $\lambda_1$ is the first eigenvalue of the operator $-\Delta u -\frac{\mu}{|x|^2}$ with zero Dirichlet boundary condition.
Ferrero and Gazzola \cite{Ferrero-Gazzola2001JDE} showed that equation \eqref{eq1.3} has at least one positive solution for  $N\geq4$, $\lambda>0$ and $\lambda\notin \sigma_\mu$; if $\mu \geq0$,  $\mu\in (\bar{\mu}-1, \bar{\mu})$ there exists $\lambda_k\in \sigma_\mu$ such that $\lambda\in (\lambda_k -S_\mu|\Omega|^{-2/N},\lambda_k)$ then equation \eqref{eq1.3} has $\nu_k$  pairs of nontrivial solutions, where  $\nu_k$ denotes the multiplicity of $\lambda_k$ and
$$S_{\mu}:=\inf_{u\in D^{1,2}(\mathbb{R}^N)\backslash \{0\} } \frac{\int_{\mathbb{R}^N}|\nabla u|^2dx -\mu\int_{\mathbb{R}^N}\frac{u^2}{x^2} dx
}{
\Big(\int_{\mathbb{R}^N}|u|^{2^*}dx\Big)^{2/2^*}}.$$
Cao and Han \cite{Cao-Han2004JDE} proved that for $N\geq5$, $\mu\in[0 ,\bar{\mu}-(\frac{N+2}{N})^2]$ and $\lambda\in (0,\lambda_1)$, equation  \eqref{eq1.3} has  a nontrivial solution.
Cao and Yan \cite{Cao-Yan2010CVPDE} showed the existence of infinitely many solutions for $N\geq5$, $\mu\in[0 ,\bar{\mu}-4]$ and $\lambda\in (0,\lambda_1)$.
If $N\geq7$ and  $\mu\in[0 ,\bar{\mu}-4]$,
Cao and Peng \cite{Cao-Peng2003JDE} proved that  equation \eqref{eq1.3} admits at least a pair of sign-changing solutions for $\lambda\in (0,\lambda_1)$.
Chen and Zou \cite{Chen-Zou2012JDE} also considered equation \eqref{eq1.3} and showed  that for $N\geq7$ and  $\mu\in[0 ,\bar{\mu}-4]$, equation \eqref{eq1.3} has infinitely many sign-changing solutions for $\lambda >0$. Furthermore, if $\lambda\in[\lambda_n,\lambda_{n+1})$ for some $n\geq1$, equation \eqref{eq1.3} has a ground state solution.
Other results can be found in \cite{Cao-Peng2003PAMS}, \cite{Catrina-Wang2001}, \cite{Chen2003JDE}, \cite{Chou-Chu1993JLMS}, \cite{Egnell1989IUMJ}, \cite{Ekeland-Ghoussoub2002BAMS}, \cite{Ferrero-Gazzola2001JDE}, \cite{Garcia-Peral1998JDE}, \cite{Ruiz-Willem2003JDE}  and the references quoted therein. We also refer the interested reader to \cite{Brezis-Nirenberg1983CPAM}, \cite{Cerami-Solimini-Struwe1986JFA}, \cite{Cerami-Zhong-Zou2015CVPDE}, \cite{Ghoussoub-Yuan2000}, \cite{Ho-Perera-Sim2023} and the discussion in the sequel concerning the results of \eqref{eq1.3} in the case of $\mu=0$ where equation \eqref{eq1.1} reduces to the  well-known Br\'{e}zis-Nirenberg problem
\begin{equation}\label{eq1.2}
\left\{\begin{array}{ll}
-\Delta u=|u|^{2^*-1}+\lambda u,    \,  &x\in  \Omega, \\
u=0,    \, &  x\in \partial \Omega,
\end{array} \right.
\end{equation}
which has been extensively studied by many researchers after,  particularly,  the  celebrated papers \cite{Brezis-Nirenberg1983CPAM}  carried out by Br\'{e}zis and Nirenberg.

The nonlocal elliptic equation \eqref{eq1.1} is related to the following nonlinear Choquard equation
\begin{equation}\label{eq1.4}
-\Delta u+V(x)u=\Big( \int_{\mathbb{R}^N}\frac{F(u(y))}{|x-y|^\alpha}dy \Big)f (u(x)),
\end{equation}
where $0<\alpha<N$ and $F$ is the primitive of $f$.
Equation \eqref{eq1.4}  arises from the study of Bose-Einstein condensation and can be used to depict the finite range many-body interactions between particles.
When $N=3$, $\alpha=1$ and $f(u)=u$, equation \eqref{eq1.4}  was proposed by  Choquard \cite{Lieb1967SAM} as an approximation to Hartree-Fock theory for a one component plasma and was also introduced by Pekar \cite{Pekar1954} to describe the the quantum theory of a polaron at rest. See also \cite{Bahrami2014}, \cite{Choquard-Stubbe-Vuffray2008DIE}, \cite{Ghimenti-Liu-Tang2023}, \cite{Ghimenti-Pagliardini2019CVPDE}, \cite{Li-Yang-Zhou2022SCM}, \cite{Moroz-Van-Schaftingen2013JFA}, \cite{Tod-Moroz1999N}, and the references therein for more mathematical and physical background as well as the existence of solutions of equation \eqref{eq1.4}.

The appearance of the nonlocal term $\Big(|x|^{-\alpha} *F(u)\Big)f(u)$ gives rise to some mathematical difficulties for the study of Choquard equation, which has received increasing attention from many authors.
Recently, Gao and Yang in  \cite{Gao-Yang2018SCM} considered equation \eqref{eq1.1} with $\mu=0$ and $f(u)=\lambda u$,  that is, the following  Br\'{e}zis-Nirenberg type problem of the nonlinear Choquard equation
\begin{equation}\label{eq1.5}
\left\{\begin{array}{ll}
-\Delta u =
\Big( \displaystyle{\int_{\Omega}} \frac{|u(y)|^{2_{\alpha}^*} }{|x-y|^{\alpha}}dy\Big) |u|^{{2_{\alpha}^*}-1}+\lambda u,     \,  &x\in  \Omega, \\
u =0,    \, &  x\in \partial \Omega,
\end{array} \right.
\end{equation}
and obtained some existence results corresponding to the well-known results in \cite{Brezis-Nirenberg1983CPAM} by exploiting the concentration
properties of the Talenti-Aubin functions.
Later, they \cite{Gao-Yang2017JMAA} also established some existence and multiplicity results for \eqref{eq1.1} with $\mu=0$. In \cite{He2022JMAA} the author proved that \eqref{eq1.5} has infinitely many solutions by applying the truncation method. Yang, Ye and Zhao \cite{Yang-Ye-Zhao2023JDE} obtain the existence and asymptotic behavior of the solutions of \eqref{eq1.5} by using the Lyapunov-Schmidt reduction procedure. Pan, Wen and Yang \cite{Pan-Wen-Yang} consider the qualitative analysis
the blow-up solutions of  \eqref{eq1.5}.

Guo and Tang \cite{Guo-Tang2025arXiv} consider the following nonlocal problem with Hardy potential
\begin{equation}\label{eq1.6}
-\Delta u -\mu\frac{u}{|x|^2}=\Big(
\int_{\mathbb{R}^N} \frac{|u(y)|^{2_{\alpha}^*} }{|x-y|^{\alpha}}dy \Big) |u|^{2_{\alpha}^*-1}, ~~x\in\mathbb{R}^N.
\end{equation}
They proved  the existence and symmetry of solutions  by developing a suitable version of concentration-compactness lemma and the moving plane method. Up to a constant,  the  solutions of  equation \eqref{eq1.6} are the minimizers of the problem:
\begin{equation*}\label{eq2.3}
S_{H,\alpha}:=\inf_{u\in D^{1,2}(\mathbb{R}^N)\backslash \{0\} } \frac{\int_{\mathbb{R}^N}|\nabla u|^2dx -\mu\int_{\mathbb{R}^N}\frac{u^2}{|x|^2} dx
}{
\Big(\int_{\mathbb{R}^N}\int_{\mathbb{R}^N} \frac{|u|^{2_{\alpha}^*}(x)|u(y)|^{2_{\alpha}^*} }{|x-y|^{\alpha}}dxdy \Big)^{1/2_{\alpha}^*}},
\end{equation*}
where $D^{1,2}(\mathbb{R}^N)=\big\{u\in L^{2^*}(\mathbb{R}^N): \nabla u\in  L^{2}(\mathbb{R}^N) \big\}$. Assume that $u_{\mu}$ is a positive solutions of equation \eqref{eq1.6}, define
$$ u_{\mu,\varepsilon} =\varepsilon^{\frac{2-N}{2} }u_{\mu}\big(\frac{x}{\varepsilon}\big),$$
then $u_{\mu,\varepsilon}$ is also the positive solution of  equation \eqref{eq1.6}. In particular, due to the homogeneity of equation \eqref{eq1.6}, $u_{\mu,\varepsilon}$  are still the minimizers for $S_{H,\alpha}$ with
\begin{equation}\label{eq1.7}
\int_{\mathbb{R}^N}|\nabla u_{\mu,\varepsilon}|^2dx -\mu\int_{\mathbb{R}^N}\frac{|u_{\mu,\varepsilon}|^2}{|x|^2} dx=
\int_{\mathbb{R}^N}\int_{\mathbb{R}^N} \frac{|u_{\mu,\varepsilon}(x)|^{2_{\mu}^*}|u_{\mu,\varepsilon}(y)|^{2_{\mu}^*} }{|x-y|^{\alpha}}dxdy= S_{H,\alpha}^{ \frac{2N-\alpha}{N-\alpha+2}}.
\end{equation}
Furthermore, Guo and Tang \cite{Guo-Tang2025arXiv} obtained the asymptotic behavior of the extremal function  $u_\mu$.
More precisely, there exists positive constants $C_1$ and $C_2$, such that
\begin{equation}\label{eq1.8}
\begin{aligned}
\frac{C_1}{\big(|x|^{\gamma'/\sqrt{\bar{\mu}}} +|x|^{\gamma/\sqrt{\bar{\mu}}}\big)^{\sqrt{\bar{\mu}}}}\leq u_\mu(x)\leq \frac{C_2}{\big(|x|^{\gamma'/\sqrt{\bar{\mu}}} +|x|^{\gamma/\sqrt{\bar{\mu}}}\big)^{\sqrt{\bar{\mu}}}},
\end{aligned}
\end{equation}
where $N\geq3$, $0<\alpha<N-(N-4)_+$, $\mu<\bar{\mu}=\frac{(N-2)^2}{4}$, $\gamma=\sqrt{\bar{\mu}} +\sqrt{\bar{\mu}-\mu}$ and $\gamma'=\sqrt{\bar{\mu}} -\sqrt{\bar{\mu}-\mu}$.

\

Our first result is about $S_{H,\alpha}$ and its counterpart on bounded domains, for which we get the following estimates.
\begin{Thm}\label{Thm1.1}
\begin{itemize}
\item[$(1)$] Let $N\geq3$, $0<\alpha<N-(N-4)_+$, $0<\mu<\bar{\mu}=\frac{(N-2)^2}{4}$, then $S_{H,\alpha}$ satisfies
$$ \frac{1}{C(N,\alpha)^{1/2_{\alpha}^*}} S_{\mu}< S_{H,\alpha}<\frac{1}{C(N,\alpha)^{1/2_{\alpha}^*}} S,$$
where
$$S:=\inf_{u\in D^{1,2}(\mathbb{R}^N)\backslash \{0\} } \frac{\int_{\mathbb{R}^N}|\nabla u|^2dx }{\Big(\int_{\mathbb{R}^N} |u|^{2^*}dx\Big)^{2/2^*}} ~~\text{and}~~S_{\mu}:=\inf_{u\in D^{1,2}(\mathbb{R}^N)\backslash \{0\} } \frac{\int_{\mathbb{R}^N}|\nabla u|^2dx -\mu\int_{\mathbb{R}^N}\frac{u^2}{x^2} dx
}{
\Big(\int_{\mathbb{R}^N}|u|^{2^*}dx\Big)^{2/2^*}}.$$
\item[$(2)$]  Let $N\geq 4$, $0<\alpha<N-(N-4)_+$. If $\mu<0$, then $S_{H,\alpha}(\Omega)=\frac{1}{C(N,\alpha)^{1/2_{\alpha}^*}}S$ and is not attained for any $\Omega$,
where
$$S_{H,\alpha}(\Omega):=\inf_{u\in D_0^{1,2}(\Omega)\backslash \{0\} } \frac{\int_{\Omega}|\nabla u|^2dx -\mu\int_{\Omega}\frac{u^2}{x^2} dx
}{
\Big(\int_{\Omega}\int_{\Omega} \frac{|u|^{2_{\alpha}^*}(x)|u(y)|^{2_{\alpha}^*} }{|x-y|^{\alpha}}dxdy \Big)^{1/2_{\alpha}^*}}.$$
\end{itemize}
\end{Thm}

\medskip

We next turn to the existence of nontrivial solution to  the following  equation,
\begin{equation}\label{eq1.9}
\left\{\begin{array}{ll}
-\Delta u -\mu\frac{u}{|x|^2}=
\Big( \displaystyle{\int_{\Omega}} \frac{|u(y)|^{2_{\alpha}^*} }{|x-y|^{\alpha}}dy\Big) |u(x)|^{{2_{\alpha}^*}-1}+\lambda u,     \,  &x\in  \Omega, \\
u(x) =0,    \, &  x\in \partial \Omega,
\end{array} \right.
\end{equation}
and we obtain the following conclusions.
\begin{Thm}\label{Thm1.2}
\begin{itemize}
\item[$(1)$] Let $N\geq3$, $0<\alpha<N$ and $0<\mu\leq\bar{\mu}$, Then, equation  \eqref{eq1.9} has a solution for all $\lambda\in (0, \lambda_1)$.
\item[$(2)$] Let $N\geq3$, $0<\alpha<N-(N-4)_+$ and $0<\mu\leq\bar{\mu}-1$, Then, equation  \eqref{eq1.9} has at least one nontrivial solution for all $\lambda\in (0, \lambda_1)$.
\end{itemize}
\end{Thm}
\par

\medskip

For the following critical Choquard equation with superlinear perturbation,
\begin{equation}\label{eq1.10}
\left\{\begin{array}{ll}
-\Delta u -\mu\frac{u}{|x|^2}=
\Big( \displaystyle{\int_{\Omega}} \frac{|u(y)|^{2_{\alpha}^*} }{|x-y|^{\alpha}}dy\Big) |u|^{{2_{\alpha}^*}-1}+\lambda u^q,     \,  &x\in  \Omega, \\
u(x) =0,    \, &  x\in \partial \Omega,
\end{array} \right.
\end{equation}
our third main result is the following:
\begin{Thm}\label{Thm1.3}
Let $N\geq3$, $0<\alpha<N-(N-4)_+$, $0<\mu<\bar{\mu}$ and $1<q<2^*-1$, Then, equation  \eqref{eq1.10} has at least one nontrivial solution provided that
\begin{itemize}
  \item $N>\frac{2(2a+q+1)}{2a+q-1}$ and  $\lambda>0$, or
  \item $N<\frac{2(2a+q+1)}{2a+q-1}$ and  $\lambda>\lambda_0$, for $\lambda_0$ large enough.
\end{itemize}
\end{Thm}
\par

\medskip

Analogously, we have an existence result for critical Choquard equation with superlinear nonlocal perturbation.
\begin{Thm}\label{Thm1.4}
Let $N\geq3$, $0<\alpha<N-(N-4)_+$, $0<\mu<\bar{\mu}$ and $1<p<2_\alpha^*-1$, Then, equation
\begin{equation}\label{eq1.11}
\left\{\begin{array}{ll}
-\Delta u -\mu\frac{u}{|x|^2}=
\Big( \displaystyle{\int_{\Omega}} \frac{|u(y)|^{2_{\alpha}^*} }{|x-y|^{\alpha}}dy\Big) |u|^{{2_{\alpha}^*}-1}+\lambda\Big( \displaystyle{\int_{\Omega}} \frac{|u(y)|^{p} }{|x-y|^{\alpha}}dy\Big) |u|^{p-1} ,     \,  &x\in  \Omega, \\
u(x) =0,    \, &  x\in \partial \Omega,
\end{array} \right.
\end{equation}
has at least one nontrivial solution provided that
\begin{itemize}
  \item $N>\frac{2a-\alpha+2p}{a+p-2}$ and  $\lambda>0$, or
  \item $N\leq\frac{2a-\alpha+2p}{a+p-2}$ and  $\lambda$ is sufficiently large.
\end{itemize}
\end{Thm}
\par

\medskip

Regarding the aforementioned three results, we primarily employ variational methods.  Associated with equation \eqref{eq1.1}, we consider the following energy functional
\begin{equation*}
I:= \frac{1}{2} \int_{\Omega} |\nabla u|^2dx -\frac{\mu }{2} \int_{\Omega} \frac{|u|^2}{|x|^2} dx -\frac{1}{2{2_{\alpha}^*}}\int_{\Omega}\int_{\Omega} \frac{|u(x)|^{2_{\alpha}^*}|u(y)|^{2_{\alpha}^*} }{|x-y|^{\alpha}}dxdy- \lambda \int_{\Omega} F(u)dx ,
\end{equation*}
where $F(u)=\int_0^uf(t)dt$. It follows that the critical points of $I$ are solutions of equation \eqref{eq1.1}.  In comparison to the previous works, the investigation of (1.1) presents significantly enhanced complexity owing to the synergistic interaction between the critical Hardy term and the upper critical growth of Choquard-type nonlinear term.  It is commonly acknowledged that the key step in addressing critical problems lies in obtaining a precise critical level associated with the sharp constant $S_{H,\alpha}$  within which one can prove a local Palais-Smale ((PS) for short) condition. Due to the absence of a specific expression for the extremal function $u_\mu(x)$, more intricate strategies are demanded to reconcile the competing effects between the Hardy potential and the non-local convolution term. It is worth emphasizing that  when $\mu=0$, Gao and Yang \cite{Gao-Yang2018SCM} have derived the specific expression for the extremal function $u_\mu(x)$ of  $S_{H,\alpha}$.
Their seminal work provides a foundational reference point for subsequent investigations, yet the generalization to $\mu\neq0$  remains an active area of research requiring innovative analytical frameworks. The present study contributes to this endeavor by the variational approach that systematically accounts for the interplay between local singular effects and non-local convolution dynamics.

The paper is organized as follows. In section \ref{sec:prelim} we collect some basic results needed in the sequel, in section \ref{sec:est} we derive the estimates about the minimization problem and in the last three sections we obtain the existence results for the different perturbation terms.

\par
\section{Notations and Preliminary results} \label{sec:prelim}
Before proving our results, we consider the work space $H^1_0(\Omega)$ which is the completion of $C_0^{\infty}(\Omega)$ with the norm
$$||u||:=\left(\int_{\Omega} |\nabla u |^2 dx \right)^{\frac{1}{2}}.$$
For all $\mu\in [0, \bar{\mu})$, we endow the scalar product of the Hilbert space $H=H_0^1(\Omega)$,
$$\langle u,v\rangle_\mu=\int_{\Omega} \nabla u \nabla vdx -\mu \int_{\Omega}\frac{uv}{|x|^{2}} dx$$
and define the norm
$$||u||^2_\mu=\langle u,v\rangle=\int_{\Omega} |\nabla u |^2dx -\mu \int_{\Omega}\frac{|u |^2}{|x|^{2}} dx.$$
By the Hardy inequality \cite{Garcia-Peral1998JDE}
$$ \int_{\Omega} \frac{|u|^2}{|x|^2}\,dx\leq \frac{1}{\bar{\mu}}\int_{\Omega} |\nabla u|^2 dx,~~\forall u\in H, $$
it is not difficult to derive that the norm $||\cdot||_\mu$ is equivalent to the usual normal  $||u||=\big(\int_{\Omega} |\nabla u |^2dx \big)^{1/2}$.
It is well-known that $H^1_0(\Omega)\hookrightarrow L^p(\Omega)$ continuously for $p\in [1, 2^*]$, compactly for $p\in [1, 2^*)$.
Moreover, write $\mathcal{L}_{\mu}:= -\Delta -\frac{\mu}{|x|^{\mu}}$ with $0< \mu< \bar{\mu}:=\frac{(N-2)^2}{4}.$ Then, according to \cite{Egnell1989IUMJ,Ekeland-Ghoussoub2002BAMS}, the spectrum $\sigma_{\mu}\subset (0,+\infty)$ of $\mathcal{L}_{\mu}$ is discrete and each eigenvalue $\lambda_{i}\, (i=1,\, 2,\, \cdots)$ is isolated and has finite multiplicity. Moreover, the smallest eigenvalue $\lambda_1$ is simple and $\lambda_i\to \infty$ as $i\to \infty.$

We will need to make use of the well-known Hardy-Littlewood-Sobolev inequality \cite{Lieb-Loss2001book}:
\begin{Pro}\label{Pro2.3}
(Hardy-Littlewood-Sobolev inequality) Let $t,r>1$, and $0<\alpha< N$ with $1/t+ 1/r+ \alpha/N=2 ,~f\in L^t(\R^N)$  and $g\in L^r(\R^N)$. There exists a sharp constant  $C(t,r,N,\mu)>0$, independent of $f$ and $g$, such that
\begin{equation}\label{eq2.1}
\int_{\mathbb{R}^N}\int_{\R^N}\frac{f(x)g(y)}{|x-y|^{\alpha}}dxdy \leq C(t,r,N,\alpha)||f||_t||g||_r.
\end{equation}
\end{Pro}
Here, $||f||_t=\big(\int_{\mathbb{R}^N} |f|^t\big)^{1/t}$. If $t=r=\frac{2N}{2N-\alpha}$, then
$$C(t,r,N,\alpha)=C(N,\alpha)= \pi^{\frac{\alpha}{2}} \frac{\Gamma(\frac{N-\alpha}{2})} {\Gamma(\frac{2N-\alpha}{2})} \bigg(\frac{\Gamma(\frac{N}{2})} {\Gamma(N)} \bigg)^{\frac{\alpha}{N}-1}.$$
In this case, the equality in \eqref{eq2.1} holds if and only if $f\equiv C g$ and
$$g(x)=A\Big(a+|x-x_0|^2\Big)^{\frac{\alpha-2N}{2}}$$
for some $A\in \mathbb{C}, x_0\in \mathbb{R}^N$ and $0\neq a \in \mathbb{R}$.

If $f(x)=h(x)=u^p(x)$ in  \eqref{eq2.1},  the following integral
$$\int_{\R^N}\int_{\R^N}\frac{|u(x)|^p|u(y)|^p
}{|x-y|^{\alpha}}dxdy $$
is well-defined provided
$$\frac{2N-\alpha}{N}\leq p \leq \frac{2N-\alpha}{N-2}. $$ $2_{*\alpha}:=\frac{2N-\alpha}{N}$ is called the lower critical exponent and $2_{\alpha}^*:=\frac{2N-\alpha}{N-2}$  the upper critical exponent in the sense of the Hardy-Littlewood-Sobolev inequality.
\par
By the Hardy-Littlewood-Sobolev inequality, for all $u \in D^{1,2}(\mathbb{R}^N)$, we get
\begin{equation}\label{eq2.2}
\int_{\R^N}\int_{\R^N} \frac{|u(x)|^{2_{\alpha}^*}|u(y)|^{2_{\alpha}^*} }{|x-y|^{\alpha}}dxdy \leq C(N,\alpha) ||u||_{2^*}^{ 2\cdot2_{\alpha}^* }<+\infty,
\end{equation}
where $C(N,\alpha)$ is defined as in Proposition \ref{Pro2.3}.

\par
\section{Estimates} \label{sec:est}
Regarding $S_{H,\alpha}$ we get the following estimate.
\begin{Lem}\label{Lem2.2}
Let $N\geq3$, $0<\alpha<N-(N-4)_+$, $0<\mu<\bar{\mu}=\frac{(N-2)^2}{4}$, then $S^*_{H,\alpha}$ defined in \eqref{eq2.3} satisfies
$$ \frac{1}{C(N,\alpha)^{1/2_{\alpha}^*}} S_{\mu}\leq S_{H,\alpha}<\frac{1}{C(N,\alpha)^{1/2_{\alpha}^*}} S,$$
where
$$S:=\inf_{u\in D^{1,2}(\mathbb{R}^N)\backslash \{0\} } \frac{\int_{\mathbb{R}^N}|\nabla u|^2dx }{\Big(\int_{\mathbb{R}^N} |u|^{2^*}dx\Big)^{2/2^*}} ~~\text{and}~~S_{\mu}:=\inf_{u\in D^{1,2}(\mathbb{R}^N)\backslash \{0\} } \frac{\int_{\mathbb{R}^N}|\nabla u|^2dx -\mu\int_{\mathbb{R}^N}\frac{u^2}{x^2} dx
}{
\Big(\int_{\mathbb{R}^N}|u|^{2^*}dx\Big)^{2/2^*}}.$$
\end{Lem}
\begin{proof}
On one side, by the Hardy-Littlewood-Sobolev inequality, we have
\begin{equation*}
\begin{split}
\Big(\int_{\mathbb{R}^N}\int_{\mathbb{R}^N} \frac{|u|^{2_{\alpha}^*}(x)|u(y)|^{2_{\alpha}^*} }{|x-y|^{\alpha}}dxdy \Big)^{1/2_{\alpha}^*}&\leq C(N,\alpha)^{1/2_{\alpha}^*}\|u\|^2_{2^*} \\
&\leq C(N,\alpha)^{1/2_{\alpha}^*}S_{\mu}^{-1} \Big( \int_{\mathbb{R}^N}|\nabla u|^2dx -\mu\int_{\mathbb{R}^N}\frac{|u|^2}{|x|^2} dx\Big),
\end{split}
\end{equation*}
which implies that
$$\frac{S_{\mu}}{C(N,\alpha)^{1/2_{\alpha}^*}}\leq S_{H,\alpha}.$$
On the other side, by \eqref{eq2.2}, it is easy to see that
\begin{equation*}
\int_{\R^N}\int_{\R^N} \frac{|\bar{U}(x)|^{2_{\alpha}^*}|\bar{U}(y)|^{2_{\alpha}^*} }{|x-y|^{\alpha}}dxdy = C(N,\alpha) ||\bar{U}||_{2^*}^{ 2\cdot2_{\alpha}^* }
\end{equation*}
if and only if
$$\bar{U}=A\Big(a+|x-x_0|^2\Big)^{\frac{2-N}{2}}$$
for some $A\in \mathbb{R}, x_0\in \mathbb{R}^N$ and $0\neq a \in \mathbb{R}$, where $\bar{U} $ is a minimizer for $S$.
It follows from $0<\mu < \bar{\mu}$ that
\begin{equation*}
\begin{aligned}
S_{H,\alpha}&=\inf_{u\in D^{1,2}(\mathbb{R}^N)\backslash \{0\} } \frac{\int_{\mathbb{R}^N}|\nabla u|^2dx -\mu\int_{\mathbb{R}^N}\frac{u^2}{x^2} dx
}{\Big(\int_{\mathbb{R}^N}\int_{\mathbb{R}^N} \frac{|u|^{2_{\alpha}^*}(x)|u(y)|^{2_{\alpha}^*} }{|x-y|^{\alpha}}dxdy \Big)^{1/2_{\alpha}^*}} \\
&\leq  \frac{\int_{\mathbb{R}^N}|\nabla \bar{U}|^2dx  -\mu\int_{\mathbb{R}^N}\frac{\bar{U}^2}{x^2} dx
}{\Big(\int_{\mathbb{R}^N}\int_{\mathbb{R}^N} \frac{|\bar{U}|^{2_{\alpha}^*}(x)|\bar{U}(y)|^{2_{\alpha}^*} }{|x-y|^{\alpha}}dxdy \Big)^{1/2_{\alpha}^*}}\\
&<\frac{1}{C(N,\alpha)^{1/2_{\alpha}^*}} \frac{\int_{\mathbb{R}^N}|\nabla \bar{U}|^2dx }{
\Big(\int_{\mathbb{R}^N}|\bar{U}|^{2^*}dx\Big)^{2/2^*}},
\end{aligned}
\end{equation*}
which means $S_{H,\alpha}<\frac{1}{C(N,\alpha)^{1/2_{\alpha}^*}} S,$
where
$$S=\frac{\int_{\mathbb{R}^N}|\nabla u|^2dx }{
\Big(\int_{\mathbb{R}^N}|u|^{2^*}dx\Big)^{2/2^*}}.$$
Therefore, $ \frac{1}{C(N,\alpha)^{1/2_{\alpha}^*}} S_{\mu}\leq S_{H,\alpha}<\frac{1}{C(N,\alpha)^{1/2_{\alpha}^*}} S.$
\end{proof}

Denote $S_{H,\alpha} (\Omega)$ as follows:
\begin{equation*}
S_{H,\alpha}(\Omega):=\inf_{u\in D_0^{1,2}(\Omega)\backslash \{0\} } \frac{\int_{\Omega}|\nabla u|^2dx -\mu\int_{\Omega}\frac{u^2}{x^2} dx
}{
\Big(\int_{\Omega}\int_{\Omega} \frac{|u|^{2_{\alpha}^*}(x)|u(y)|^{2_{\alpha}^*} }{|x-y|^{\alpha}}dxdy \Big)^{1/2_{\alpha}^*}}.
\end{equation*}
Regarding $S_{H,\alpha} (\Omega)$ we get the following estimate.
\begin{Lem}
Let $N\geq 4$, $0<\alpha<N-(N-4)_+$. If $\mu<0$, then $S_{H,\alpha}(\Omega)=\frac{1}{C(N,\alpha)^{1/2_{\alpha}^*}}S$ and is not attained for any $\Omega$,
where
$$S:=\inf_{u\in D^{1,2}(\mathbb{R}^N)\backslash \{0\} } \frac{\int_{\mathbb{R}^N}|\nabla u|^2dx }{\Big(\int_{\mathbb{R}^N} |u|^{2^*}dx\Big)^{2/2^*}} .$$
\end{Lem}
\begin{proof}
Assume that $U(x)=\frac{[N(N-2)]^{\frac{N-2}{4}}}{(1+|x|^2)^{\frac{N-2}{2}}}$.
Let $\phi\in C_0^\infty(\Omega)$ be a negative function such that
$\phi(x)=1$, $\forall x\in B(0,\delta)$, $\phi(x)=0$, $\forall x\in \Omega\setminus B(0,2\delta)$ and $0\leq \phi(x)\leq 1$, $\forall x\in \Omega$, where $2\delta<|\xi|$ and $B(0,3|\xi|)\subset\Omega$.
For $\varepsilon>0$, we define,
$$\begin{aligned}
u_\varepsilon(x):=\psi(x)U_\varepsilon(x) ~~\text{where}~~U_\varepsilon(x) :=\varepsilon^{\frac{2-N}{2}}U\Big(\frac{x}{\varepsilon}\Big).\end{aligned}$$
It follows from \cite{Gao-Yang2018SCM} that, as $\varepsilon\to0^+$,
$$\int_\Omega|\nabla u_\varepsilon|^2dx=S^{\frac N2}+O(\varepsilon^{N-2})=C(N,\alpha)^{\frac{N(N-2)}{2(2N-\alpha)}}S_{H,L}^{\frac N2}+O(\varepsilon^{N-2})  ,$$
$$\int_\Omega|u_\varepsilon|^2dx = \begin{cases} O(\varepsilon^2|\ln\varepsilon|),&\text{ if }\quad N=4,\\O(\varepsilon^2),&\text{ if }\quad N\geqslant5,\end{cases}$$
and
$$\displaystyle{\int_{\Omega}}\displaystyle{\int_{\Omega}} \frac{|u_\varepsilon(x)|^{2_{\alpha}^*}|u_\varepsilon(y)|^{2_{\alpha}^*} }{|x-y|^{\alpha}}dxdy \geq C(N,\alpha)^{\frac{N}{2}} S_{H,L}^{\frac{2N-\alpha}{2}} -O(\varepsilon^{\frac{2N-\alpha}{2}}) = C(N,\alpha) S^{\frac{2N-\alpha}{2}} -O(\varepsilon^{\frac{2N-\alpha}{2}}).$$
Denote $v_\varepsilon(x)=u_\varepsilon(x+2\xi)$, by the invariance by translation of the integral,
$$\int_{\Omega}\frac{|v_\varepsilon(x)|^2}{|x|^2}= \int_{\Omega}\frac{|u_\varepsilon(x+2\xi)|^2}{|x|^2}\leq\frac{1}{|\xi|^2} \int_{\Omega}|u_\varepsilon(x)|^2 .$$
Combining this and $\mu<0$,
\begin{equation*}
\begin{aligned}
S_{H,\alpha}(\Omega)&\leq \frac{\int_{\Omega}|\nabla v_\varepsilon|^2dx -\mu\int_{\Omega}\frac{v_\varepsilon^2}{x^2} dx
}{\Big(\int_{\Omega}\int_{\Omega} \frac{|v_\varepsilon(x)|^{2_{\alpha}^*}|v_\varepsilon(y)|^{2_{\alpha}^*} }{|x-y|^{\alpha}}dxdy \Big)^{1/2_{\alpha}^*}}\\
&\leq \frac{\int_{\Omega}|\nabla u_\varepsilon(x)|^2dx -\frac{\mu}{|\xi|^2} \int_{\Omega}|u_\varepsilon(x)|^2dx
}{\Big(\int_{\Omega}\int_{\Omega} \frac{|u_\varepsilon(x)|^{2_{\alpha}^*}|u_\varepsilon(y)|^{2_{\alpha}^*} }{|x-y|^{\alpha}}dxdy \Big)^{1/2_{\alpha}^*}}.
\end{aligned}
\end{equation*}
Now we distinguish the following cases:
\begin{itemize}
\item[(\romannumeral1)] In the case $N=4$,
\begin{equation*}
\begin{aligned}
S_{H,\alpha}(\Omega)&\leq \frac{\int_{\Omega}|\nabla u_\varepsilon(x)|^2dx -\frac{\mu}{|\xi|^2} \int_{\Omega}|u_\varepsilon(x)|^2dx
}{\Big(\int_{\Omega}\int_{\Omega} \frac{|u_\varepsilon(x)|^{2_{\alpha}^*}|u_\varepsilon(y)|^{2_{\alpha}^*} }{|x-y|^{\alpha}}dxdy \Big)^{1/2_{\alpha}^*}}\\
&\leq\frac{ S^{2}+O(\varepsilon^{2}) +O(\varepsilon^2|\ln\varepsilon|) }{\Big(C(N,\alpha) S^{\frac{8-\alpha}{2}} -O(\varepsilon^{\frac{8-\alpha}{2}}) \Big)^{\frac{2}{8-\alpha}}}\\
&\leq \frac{1}{C(N,\alpha)^{\frac{2}{8-\alpha}}}S.
\end{aligned}
\end{equation*}

\item[(\romannumeral2)] In the case $N\geq5 $,
\begin{equation*}
\begin{aligned}
S_{H,\alpha}(\Omega)&\leq \frac{\int_{\Omega}|\nabla u_\varepsilon(x)|^2dx -\frac{\mu}{|\xi|^2} \int_{\Omega}|u_\varepsilon(x)|^2dx
}{\Big(\int_{\Omega}\int_{\Omega} \frac{|u_\varepsilon(x)|^{2_{\alpha}^*}|u_\varepsilon(y)|^{2_{\alpha}^*} }{|x-y|^{\alpha}}dxdy \Big)^{1/2_{\alpha}^*}}\\
&\leq\frac{ S^{\frac N2}+O(\varepsilon^{N-2})+O(\varepsilon^2) }{\Big(C(N,\alpha) S^{\frac{2N-\alpha}{2}} -O(\varepsilon^{\frac{2N-\alpha}{2}}) \Big)^{1/2_{\alpha}^*}}\\
&\leq \frac{1}{C(N,\alpha)^{1/2_{\alpha}^*}}S.
\end{aligned}
\end{equation*}
\end{itemize}
Thus, $S_{H,\alpha}(\Omega)\leq \frac{1}{C(N,\alpha)^{1/2_{\alpha}^*}}S$. We deduce  from \cite{Gao-Yang2018SCM} and $\mu<0$ that $\frac{1}{C(N,\alpha)^{1/2_{\alpha}^*}}S \leq S_{H,\alpha}(\Omega)$. Then $ S_{H,\alpha}(\Omega)=\frac{1}{C(N,\alpha)^{1/2_{\alpha}^*}}S$. Moreover, the infmum $S_{H,\alpha}(\Omega)$ cannot be attained because such a minimizer would contradict the fact that $S$  is never achieved except when $\Omega=\mathbb{R}^N$.
\end{proof}

\section{Linear perturbation}
In this section, we consider the existence  of the solutions to the following equation
\begin{equation}\label{eq3.1}
\left\{\begin{array}{ll}
-\Delta u(x)-\mu\frac{u}{|x|^2}=
\Big( \displaystyle{\int_{\Omega}} \frac{|u(y)|^{2_{\alpha}^*} }{|x-y|^{\alpha}}dy\Big) |u|^{{2_{\alpha}^*}-1}+\lambda u,     \,  &x\in  \Omega, \\
u(x) =0,    \, &  x\in \partial \Omega.
\end{array} \right.
\end{equation}
The energy functional corresponding to \eqref{eq3.1} is
\begin{equation*}
\mathcal{I}(u) := \frac{1}{2} \int_{\Omega} |\nabla u|^2dx -\frac{\mu }{2} \int_{\Omega} \frac{|u|^2}{|x|^2} dx-  \frac{\lambda}{2} \int_{\Omega} | u|^{2}dx -\frac{1}{2\cdot{2_{\alpha}^*}}\int_{\Omega}\int_{\Omega} \frac{|u(x)|^{2_{\alpha}^*}|u(y)|^{2_{\alpha}^*} }{|x-y|^{\alpha}}dxdy.
\end{equation*}

It is not difficult to check that the functional $\mathcal{I}$ satisfies the Mountain-Pass Geometry, that is

\par
\begin{Lem}\label{Lem3.1}
The functional $\mathcal{I}$ has the mountain pass geometry:
\begin{itemize}
\item[$(1)$]  There exist $\beta, \rho>0$ such that $\mathcal{I}(u)\geq \beta$ whenever $||u||=\rho$.
\item[$(2)$]  There is an $e\in H_0^1(\Omega)$ with $||e||\geq \rho$ such that $\mathcal{I}(e)<0$.
\end{itemize}
\end{Lem}
\begin{proof}
$(1)$ By $0<\mu < \bar{\mu}$, the Sobolev embedding and the Hardy-Littlewood-Sobolev inequality, for all
$u\in H_0^1(\Omega)\setminus\{0\}$ we have
$$\begin{aligned}
\mathcal{I}(u)&=\frac{1}{2}\|u\|_\mu^2  -  \frac{\lambda}{2} \int_{\Omega} | u|^2dx -\frac{1}{2\cdot{2_{\alpha}^*}}\int_{\Omega}\int_{\Omega} \frac{|u(x)|^{2_{\alpha}^*}|u(y)|^{2_{\alpha}^*} }{|x-y|^{\alpha}}dxdy\\
&\geq \Big(\frac{1}{2}- \frac{\lambda}{2\lambda_1} \Big)\|u\|_\mu^2   -\frac{1}{2\cdot{2_{\alpha}^*}}\int_{\Omega}\int_{\Omega} \frac{|u(x)|^{2_{\alpha}^*}|u(y)|^{2_{\alpha}^*} }{|x-y|^{\alpha}}dxdy\\
&\geq \frac{1}{2} \Big( 1- \frac{\lambda}{\lambda_1}\Big) \|u\|_\mu^2    -\frac1{2\cdot2_\alpha^*}C_2\|u\|^{2(\frac{2N-\alpha}{N-2})}_\mu.
\end{aligned}$$
It follows from $2<2(\frac{2N-\alpha}{N-2})$ that  we can choose some $\alpha,\rho>0$ such that $\mathcal{I}(u)\geqslant\alpha$ for $\|u\|=\rho.$

$(2)$  For some $u_0\in H_0^1(\Omega)\setminus\{0\}$, we have
$$\begin{aligned}
\mathcal{I}(tu_0)&=\frac{t^2}{2}\|u_0\|_\mu^2-\frac{\lambda t^2}{2}\int_\Omega |u_0|^2dx-\frac{t^{2\cdot2_\alpha^*}}{2\cdot2_\alpha^*} \int_\Omega\int_\Omega\frac{|u_0(x)|^{2_\alpha^*}|u_0(y)|^{2_\alpha^*}}{|x-y|^\alpha}dxdy
\end{aligned}$$
for $t > $ large enough. Then we may take an $e:= t_*u_0$ for some $t_*0 > 0 $ and  $(2)$  follows.
\end{proof}

\begin{Lem}\label{Lem3.2}
Let $0<\mu < \bar{\mu}$, $0<\lambda<\lambda_1$ and $\{u_n\}$ is a $(PS)_c$ sequence of $\mathcal{I}$, then ${u_n}$ is bounded. Moreover, assume that $u_0\in H_0^1(\Omega)$ is the weak limit of ${u_n}$, then $u_0$ is a weak solution of equation \eqref{eq3.1}.
\end{Lem}
\begin{proof}
Suppose that $\{u_n\}$ is a $(PS)_c$ sequence of $\mathcal{I}$, that is,
\begin{equation*}\label{eq}
\mathcal{I}(u_n)\to c~~\text{ and}~~~\mathcal{I}'(u_n)\to 0~\text{in }~ (H_0^1(\Omega))^{-1}.
\end{equation*}
By $0<\mu < \bar{\mu}$, $0<\lambda<\lambda_1$  and $2<2(\frac{2N-\alpha}{N-2})$, we have that, for some $C_0>0$
$$\begin{aligned}
c+o_n(1)&\geq \mathcal{I}(u_n) -\frac{1}{2\cdot{2_{\alpha}^*}}\langle \mathcal{I}'(u_n), u_n\rangle\\
&=\Big(\frac{1}{2}-\frac{1}{2\cdot{2_{\alpha}^*}} \Big) \Big(\int_{\Omega} | u|^2dx-\mu\int_{\Omega} \frac{|u|^2}{|x|^2} dx -\lambda\int_{\Omega} | u|^2dx \Big)\\
&\geq \Big(\frac{1}{2}-\frac{1}{2\cdot{2_{\alpha}^*}} \Big)C_0 ||u||^2,
\end{aligned}$$
which means that  $\{u_n\}$ is bounded in $H_0^1(\Omega)$.
Up to
a subsequence, there exists $u_0\in H_0^1(\Omega)$ such that $u_n\rightharpoonup u_0$ in $H_0^1(\Omega)$,  $u_n\rightharpoonup  u_0$ in $L^{2}(\Omega, |x|^{-2}dx)$,  $u_n\rightharpoonup  u_0$ in $L^{2^*}(\Omega)$ and  $u_n\to u_0$ a.e. in $\Omega$ as $n\to +\infty$. Then $|u_n|^{2_{\alpha}^*}\rightharpoonup  |u_0|^{2_{\alpha}^*}$ in $L^{\frac{2N}{2N-\alpha}}(\Omega)$ as $n\to +\infty$.
Since the Riesz potential of $u$  defined by
$$I_\alpha(u)(x)=|x|^{-\alpha}\ast u=\int_{\Omega}\frac{u(y)}{|x-y|^{\alpha}}dy$$
is a linear continuous map from $L^{\frac{2N}{2N-\alpha}}(\Omega)$ to $L^{\frac{2N}{\alpha}}(\Omega)$, by the Hardy-Littlewood-Sobolev inequality, we know that
$$\int_{\Omega}\frac{|u_n(y)|^{2_\alpha^*}}{|x-y|^{\alpha}}dy \rightharpoonup \int_{\Omega}\frac{|u_0(y)|^{2_\alpha^*}}{|x-y|^{\alpha}}dy~~\text{in}~ L^{\frac{2N}{-\alpha}}(\Omega)$$
as $n\to +\infty$.   Combining  this with the following fact that, $n\to +\infty$,
$$|u_n|^{2_{\alpha}^*-2}u_n\rightharpoonup  |u_0|^{2_{\alpha}^*-2}u_0~~\text{in}~L^{\frac{2N}{2N-\alpha}}(\Omega),$$
we have that
$$\int_{\Omega}\frac{|u_n(y)|^{2_\alpha^*}}{|x-y|^{\alpha}}dy |u_n|^{2_{\alpha}^*-2}u_n\rightharpoonup \int_{\Omega}\frac{|u_0(y)|^{2_\alpha^*}}{|x-y|^{\alpha}}dy |u_0|^{2_{\alpha}^*-2}u_0~~\text{in}~L^{\frac{2N}{N+2}}(\Omega)$$
as $n\to +\infty$. Since $\mathcal{I}'(u_n)\to 0$ in  $(H_0^1(\Omega))^{-1}$, for any $\phi\in H_0^1(\Omega)$ we have
$$o_n(1)= \int_{\Omega} \nabla u_n \nabla \phi dx - \int_{\Omega} \frac{u_n \phi}{|x|^2} dx -\lambda\int_{\Omega} u_n \phi dx
-\int_{\Omega} \int_{\Omega}\frac{|u_n(y)|^{2_\alpha^*} |u_n(x)|^{2_{\alpha}^*-2}u_n(x)\phi(x) }{|x-y|^{\alpha}}dxdy.$$
Thus
$$\int_{\Omega} \nabla u_0 \phi dx - \int_{\Omega} \frac{u_0 \phi}{|x|^2} dx -\lambda\int_{\Omega} u_0 \phi dx =\int_{\Omega} \int_{\Omega}\frac{|u_0(y)|^{2_\alpha^*} |u_0(x)|^{2_{\alpha}^*-2}u_0(x)\phi(x) }{|x-y|^{\alpha}}dxdy,~~\forall \phi \in H_0^1(\Omega),$$
which yields that $u_0$ is a weak solution of the equation \eqref{eq3.1}.
\end{proof}

\par
\begin{Lem}\label{Lem3.3}
Let $0<\mu < \bar{\mu}$, $0<\lambda<\lambda_1$ and ${u_n}$ is a $(PS)_c$ sequence of $\mathcal{I}$ with
$$c< \frac{N+2-\alpha}{2(2N-\alpha)}S_{H,\mu}^{\frac{2N-\alpha}{N+2-\alpha }},$$
then $\{u_n\}$ has a convergent subsequence in  $H_0^1(\Omega)$.
\end{Lem}
\begin{proof}
Denote $v_n:=u_n- u_0$, combining  $u_n\rightharpoonup u_0$ in $H_0^1(\Omega)$ with the fact that the embedding $H_0^1(\Omega)\hookrightarrow L^q(\Omega)$ is continuous for   $q\in [1, 2^*]$ and compact for $q\in [1, 2^*)$,
$v_n\rightharpoonup  0$ in $L^{2^*}(\Omega)$, $v_n\to 0$ in $L^q(\Omega)$ for $q\in [1, 2^*)$ and  $v_n\to 0$ a.e. in $\Omega$ as $n\to +\infty$.
By the Brezis-Lieb Lemma
$$\int_{\Omega} |\nabla u_n|^2dx=\int_{\Omega} |\nabla v_n|^2dx + \int_{\Omega} |\nabla u_0|^2dx+o_n(1),$$
$$\int_{\Omega} \frac{|u_n|^2}{|x|^2}dx=\int_{\Omega}\frac{|v_n|^2}{|x|^2}dx + \int_{\Omega} \frac{|u_0|^2}{|x|^2}dx+o_n(1)$$
and
$$\int_{\Omega}\int_{\Omega}\frac{|u_n(y)|^{2_\alpha^*} |u_n(x)|^{2_{\alpha}^*} }{|x-y|^{\alpha}}dxdy=\int_{\Omega}\int_{\Omega}\frac{|v_n(y)|^{2_\alpha^*} |v_n(x)|^{2_{\alpha}^*} }{|x-y|^{\alpha}}dxdy +\int_{\Omega}\int_{\Omega}\frac{|u_0(y)|^{2_\alpha^*} |u_0(x)|^{2_{\alpha}^*} }{|x-y|^{\alpha}}dxdy +o_n(1).$$
Since $\mathcal{I}(u_0)\geq0$ and  $u_n\to u_0$ in $L^q(\Omega)$ for $q\in [1, 2^*)$, we have
\begin{equation}\label{eq3.2}
\begin{aligned}
c\leftarrow \mathcal{I}(u_n)
&= \frac{1}{2} \int_{\Omega} |\nabla v_n|^2dx +\frac{1}{2} \int_{\Omega} |\nabla u_0|^2dx -\frac{\mu }{2} \int_{\Omega} \frac{|v_n|^2}{|x|^2} dx -\frac{\mu }{2} \int_{\Omega} \frac{|u_0|^2}{|x|^2} dx -\frac{\lambda}{2} \int_{\Omega} | u_0|^2dx \\
&-\frac{1}{2\cdot{2_{\alpha}^*}}\int_{\Omega}\int_{\Omega} \frac{|v_n(x)|^{2_{\alpha}^*}|v_n(y)|^{2_{\alpha}^*} }{|x-y|^{\alpha}}dxdy -\frac{1}{2\cdot{2_{\alpha}^*}}\int_{\Omega}\int_{\Omega} \frac{|u_0(x)|^{2_{\alpha}^*}|u_0(y)|^{2_{\alpha}^*} }{|x-y|^{\alpha}}dxdy\\
&=\mathcal{I}(u_0)+\frac{1}{2} \int_{\Omega} |\nabla v_n|^2dx-\frac{\mu }{2} \int_{\Omega} \frac{|v_n|^2}{|x|^2} dx-\frac{1}{2\cdot{2_{\alpha}^*}}\int_{\Omega}\int_{\Omega} \frac{|v_n(x)|^{2_{\alpha}^*}|v_n(y)|^{2_{\alpha}^*} }{|x-y|^{\alpha}}dxdy\\
&\geq \frac{1}{2} \int_{\Omega} |\nabla v_n|^2dx-\frac{\mu }{2} \int_{\Omega} \frac{|v_n|^2}{|x|^2} dx-\frac{1}{2\cdot{2_{\alpha}^*}}\int_{\Omega}\int_{\Omega} \frac{|v_n(x)|^{2_{\alpha}^*}|v_n(y)|^{2_{\alpha}^*} }{|x-y|^{\alpha}}dxdy.
\end{aligned}
\end{equation}
It follows from $\langle \mathcal{I}'(u_0), u_0\rangle =0$ that

\begin{equation}\label{eq3.3}
\begin{aligned}
o_n(1)&=\langle \mathcal{I}'(u_n), u_n\rangle \\
&=\int_{\Omega} |\nabla v_n|^2dx +\int_{\Omega} |\nabla u_0|^2dx -\int_{\Omega} \frac{|v_n|^2}{|x|^2} dx -\int_{\Omega} \frac{|u_0|^2}{|x|^2} dx+\lambda \int_{\Omega} | u_0|^2dx\\
&-\int_{\Omega}\int_{\Omega} \frac{|v_n(x)|^{2_{\alpha}^*}|v_n(y)|^{2_{\alpha}^*} }{|x-y|^{\alpha}}dxdy -\int_{\Omega}\int_{\Omega} \frac{|u_0(x)|^{2_{\alpha}^*}|u_0(y)|^{2_{\alpha}^*} }{|x-y|^{\alpha}}dxdy\\
&=\langle \mathcal{I}'(u_0), u_0\rangle +\int_{\Omega} |\nabla v_n|^2dx -\int_{\Omega} \frac{|v_n|^2}{|x|^2} dx-\int_{\Omega}\int_{\Omega} \frac{|v_n(x)|^{2_{\alpha}^*}|v_n(y)|^{2_{\alpha}^*} }{|x-y|^{\alpha}}dxdy.
\end{aligned}
\end{equation}
From \eqref{eq3.3}, we know there exists a non-negative constant $m$ such that
$$\int_{\Omega} |\nabla v_n|^2dx -\int_{\Omega} \frac{|v_n|^2}{|x|^2} dx\to m$$
and
$$\int_{\Omega}\int_{\Omega} \frac{|v_n(x)|^{2_{\alpha}^*}|v_n(y)|^{2_{\alpha}^*} }{|x-y|^{\alpha}}dxdy \to m$$
as $n \to +\infty$. By \eqref{eq3.2} we have
\begin{equation}\label{eq3.4}
\begin{aligned}
c\geq \frac{N-\alpha+2}{2(2N-\alpha)}m.
\end{aligned}
\end{equation}
We claim that $m=0$. Otherwise, by the definition of the best constant $S_{H,\mu}$, we obtain
$$S_{H,\mu}\Big(\int_{\mathbb{R}^N}\int_{\mathbb{R}^N} \frac{|v_n|^{2_{\alpha}^*}(x)|v_n(y)|^{2_{\alpha}^*} }{|x-y|^{\alpha}}dxdy \Big)^{1/2_{\alpha}^*}\leq \int_{\mathbb{R}^N}|\nabla v_n|^2dx -\mu\int_{\mathbb{R}^N}\frac{v_n^2}{|x|^2} dx,$$
which means $m \geq S_{H,\mu}^{\frac{2N-\alpha}{N+2-\alpha}}$. Then we deduce from \eqref{eq3.4} that
$$\frac{N-\alpha+2}{2(2N-\alpha)}S_{H,\mu}^{\frac{2N-\alpha}{N+2-\alpha}}\leq \frac{N-\alpha+2}{2(2N-\alpha)}m\leq c,$$
which contradicts with the fact that
$$c< \frac{N+2-\alpha}{2(2N-\alpha)}S_{H,\mu}^{\frac{2N-\alpha}{N+2-\alpha}}. $$
Thus $m=0$ and
$$||v_n||^2=\int_{\Omega} |\nabla v_n|^2dx -\int_{\Omega} \frac{|v_n|^2}{|x|^2} dx\to 0$$
as $n\to +\infty$.  Therefore, $||u_n||\to ||u_0||$ as $n\to +\infty$.
\end{proof}

\begin{Lem}\label{Lem3.4}
Let $0<\mu \leq \bar{\mu}-1$, then  there exists $\varepsilon>0$ small enough such that
$$\max_{t\geq0}\mathcal{I}(t\bar{u}_\varepsilon)< \frac{N-\alpha+2}{2(2N-\alpha)}S_{H, \mu}^{\frac{2N-\alpha}{N+2-\alpha}}.$$
\end{Lem}
\begin{proof}
Let us fix $\delta>0$ small enough such that $B_{4\delta}(0)\subset \Omega $ and let $\eta\in C_0^\infty(\R^N)$ be a cut-off function such that $0\leq \eta \leq 1$ in $\R^N$, $\eta = 1$ in $B_{\delta}(0)$,  $\eta = 0$ in $B_{2\delta}^c(0)$.

Define
$$\bar{u}_\varepsilon ( x) = \eta ( x) u_{\mu,\varepsilon} ( x), ~~\text{where}~~ u_{\mu,\varepsilon} =\varepsilon^{\frac{2-N}{2} }u_{\mu}\big(\frac{x}{\varepsilon}\big).$$
Note that $u_{\mu,\varepsilon}$ is also a positive solution of \eqref{eq3.1} and also a minimizer for $S_{H,\alpha}$
\begin{equation}\label{eq3.5}
\int_{\mathbb{R}^N}|\nabla u_{\mu,\varepsilon}|^2dx -\mu\int_{\mathbb{R}^N}\frac{u_{\mu,\varepsilon}^2}{|x|^2} dx=
\int_{\mathbb{R}^N}\int_{\mathbb{R}^N} \frac{|u_{\mu,\varepsilon}(x)|^{2_{\alpha}^*}|u_{\mu,\varepsilon}(y)|^{2_{\alpha}^*} }{|x-y|^{\alpha}}dxdy= S_{H,\alpha}^{ \frac{2N-\alpha}{N-\alpha+2}}.
\end{equation}
Set
$a=\sqrt{1-\frac{4\mu}{(N-2)^2}}$, then $\frac{\gamma'}{\sqrt{\bar{\mu}}}=1-a$ and  $\frac{\gamma}{\sqrt{\bar{\mu}}}=1+a$. In the spirit of  \cite{Brezis-Nirenberg1983CPAM}, we have
\begin{equation}\label{eq3.6}
\int_\Omega\left(|\nabla \bar{u}_\varepsilon|^2-\mu\frac{\bar{u}_\varepsilon^2}{|x|^2}\right)\mathrm{d}x \leq S_{H,\mu}^{ \frac{2N-\alpha}{N-\alpha+2}}+O(\varepsilon^{(N-2)a}) .\end{equation}
Indeed, by \eqref{eq3.5} we get that
\begin{equation*}
\begin{aligned}
&\int_\Omega\left(|\nabla \bar{u}_\varepsilon|^2-\mu\frac{\bar{u}_\varepsilon^2}{|x|^2}\right)dx\\
&=\int_\Omega \eta^2|\nabla u_{\mu,\varepsilon}|^2+ \int_\Omega |\nabla\eta|^2 u_{\mu,\varepsilon}^2 dx +2 \int_\Omega \eta u_{\mu,\varepsilon}\big(\nabla\eta\cdot\nabla u_{\mu,\varepsilon}\big)dx-\mu\int_\Omega\frac{\eta^2u_{\mu,\varepsilon}^2}{|x|^2}dx\\
&=\int_{\R^N}\int_{\R^N}\frac{|u_{\mu,\varepsilon}(y)|^{2_\alpha^*}| |u_{\mu,\varepsilon}(x)|^{2_\alpha^*}\eta^2(x)}{|x-y|^\alpha} dxdy +\mu\int_{\R^N}\frac{\eta^2u_{\mu,\varepsilon}^2}{|x|^2}dx+ \int_\Omega |\nabla\eta|^2 u_{\mu,\varepsilon}^2 dx-\mu\int_\Omega\frac{\eta^2u_{\mu,\varepsilon}^2}{|x|^2}dx\\
&\leq\int_{\R^N}\int_{\R^N}\frac{|u_{\mu,\varepsilon}(y)|^{2_\alpha^*}| |u_{\mu,\varepsilon}(x)|^{2_\alpha^*}}{|x-y|^\alpha}dxdy + \int_\Omega |\nabla\eta|^2 u_{\mu,\varepsilon}^2 dx\\
&=S_{H,\mu}^{ \frac{2N-\alpha}{N-\alpha+2}} +\int_\Omega |\nabla\eta|^2 u_{\mu,\varepsilon}^2 dx\\
&:=S_{H,\mu}^{ \frac{2N-\alpha}{N-\alpha+2}} +\mathbb{A},
\end{aligned}
\end{equation*}
where
\begin{equation*}
\begin{aligned}
\mathbb{A}&:=\int_\Omega |\nabla\eta|^2 u_{\mu,\varepsilon}^2 dx\\
&\leq C\varepsilon^{(2-N)} \int_{B_{2\delta} \setminus B_{\delta}}\frac{1}{\big( |\frac{x}{\varepsilon}|^{1-a}+|\frac{x}{\varepsilon}|^{1+a} \big)^{N-2}}dx\\
&=C\varepsilon^{(N-2)a}  \int_{B_{2\delta} \setminus B_{\delta}}\frac{1}{\big( |x|^{1-a}(\varepsilon^{2a}+ |x|^{2a} )\big)^{N-2}}dx\\
&\leq C\varepsilon^{(N-2)a}  \int_{B_{2\delta} \setminus B_{\delta}} \frac{1}{ |x|^{(1+a)(N-2)}}dx\\
&=O(\varepsilon^{(N-2)a}).
\end{aligned}
\end{equation*}
On the one hand,
\begin{equation*}
\begin{aligned}
&\int_\Omega\int_\Omega\frac{|\bar{u}_\varepsilon(x)|^{2_\alpha^*} |\bar{u}_\varepsilon(y)|^{2_\alpha^*}}{|x-y|^\alpha}dxdy\\
& \geqslant\int_{B_\delta}\int_{B_\delta}\frac{|\bar{u}_\varepsilon(x) |^{2_\alpha^*}|\bar{u}_\varepsilon(y)|^{2_\alpha^*}}{|x-y|^\alpha}dxdy \\
&=\int_{B_\delta}\int_{B_\delta} \frac{|u_{\mu,\varepsilon}(x)|^{2_\alpha^*}|u_{\mu,\varepsilon}(y) |^{2_\alpha^*}}{|x-y|^\alpha}dxdy\\
&=\int_{\mathbb{R}^N}\int_{\mathbb{R}^N}\frac{|u_{\mu,\varepsilon}(x)|^{2_\alpha^*} |u_{\mu,\varepsilon}(y)|^{2_\alpha^*}}{|x-y|^\alpha}dxdy -2\int_{\mathbb{R}^N\setminus B_\delta}\int_{B_\delta}\frac{|u_{\mu,\varepsilon}(x)|^{2_\alpha^*} |u_{\mu,\varepsilon}(y)|^{2_\alpha^*}}{|x-y|^\alpha}dxdy \\
& -\int_{\mathbb{R}^N\setminus B_\delta}\int_{\mathbb{R}^N\setminus B_\delta}\frac{|u_{\mu,\varepsilon}(x)|^{2_\alpha^*} |u_{\mu,\varepsilon}(y)|^{2_\alpha^*}}{|x-y|^\alpha}dxdy \\
&:=S_{H,\mu}^{ \frac{2N-\alpha}{N-\alpha+2}}-2\mathbb{B}-\mathbb{C},
\end{aligned}
\end{equation*}
where
$$\mathbb{B}  :=\int_{\mathbb{R}^N\setminus B_\delta}\int_{B_\delta}\frac{|u_{\mu,\varepsilon}(x) |^{2_\alpha^*}|u_{\mu,\varepsilon}(y)|^{2_\alpha^*}}{|x-y|^\alpha}dxdy$$
and
$$\mathbb{C}:=\int_{\mathbb{R}^N\setminus B_\delta}\int_{\mathbb{R}^N\setminus B_\delta}\frac{|u_{\mu,\varepsilon}(x)|^{2_\alpha^*} |u_{\mu,\varepsilon}(y)|^{2_\alpha^*}}{|x-y|^\alpha}dxdy.$$
After a simple calculation, we have
$$\begin{aligned}
\mathbb{B}&\leq\int_{\mathbb{R}^N\setminus B_\delta}\int_{B_\delta}\frac{C\varepsilon^{(\alpha-2N)}}{\big(|\frac{x}{\varepsilon}|^{1-a} +|\frac{x}{\varepsilon}|^{1+a} \big)^{\frac{2N-\alpha}{2}}|x-y|^\alpha \big(|\frac{y}{\varepsilon}|^{1-a} +|\frac{y}{\varepsilon}|^{1+a} \big)^{\frac{2N-\alpha}{2}}}dxdy\\
&\leq C\varepsilon^{(\alpha-2N)} \bigg(\int_{\mathbb{R}^N\setminus B_\delta}\frac{1}{\big(|\frac{x}{\varepsilon}|^{1-a} +|\frac{x}{\varepsilon}|^{1+a}  \big)^N}dx\bigg)^{\frac{2N-\alpha}{2N}} \bigg(\int_{B_\delta}\frac{1}{\big(|\frac{y}{\varepsilon}|^{1-a} +|\frac{y}{\varepsilon}|^{1+a}  \big)^N}dy\bigg)^{\frac{2N-\alpha}{2N}}\\
&= C\varepsilon^{(\alpha-2N)} \bigg(\int_{\mathbb{R}^N\setminus B_\delta}\frac{\varepsilon^{(1+a)N}}{\big(|x|^{1-a}(\varepsilon^{2a} +|x|^{2a})   \big)^N}dx\bigg)^{\frac{2N-\alpha}{2N}} \bigg(\int_{B_\delta}\frac{1}{\big(|\frac{y}{\varepsilon}|^{1-a} +|\frac{y}{\varepsilon}|^{1+a}  \big)^N}dy\bigg)^{\frac{2N-\alpha}{2N}}\\
&\leq C\varepsilon^{\frac{(2N-\alpha)a}{2}} \bigg(\int_{|x|>\delta}\frac{1}{|x|^{(1+a)N}   }dx\bigg)^{\frac{2N-\alpha}{2N}} \bigg(\int_{|x|<\frac{\delta}{\varepsilon}}\frac{1}{\big(|x|^{1-a} +|x|^{1+a}  \big)^N}dx\bigg)^{\frac{2N-\alpha}{2N}}\\
&\leq C\varepsilon^{\frac{(2N-\alpha)a}{2}}\bigg(\int_{0}^{\frac{\delta}{ \varepsilon}}\frac{r^{N-1}}{\big(|r|^{1-a} +|r|^{1+a}  \big)^N}dr\bigg)^{\frac{2N-\alpha}{2N}}\\
&=O(\varepsilon^{\frac{(2N-\alpha)a}{2}})
\end{aligned}$$
and
$$\begin{aligned}
\mathbb{C}&\leq\int_{\mathbb{R}^N\setminus B_\delta}\int_{\mathbb{R}^N\setminus B_\delta}\frac{C\varepsilon^{(\alpha-2N)}}{\big(|\frac{x}{\varepsilon}|^{1-a} +|\frac{x}{\varepsilon}|^{1+a} \big)^{\frac{2N-\alpha}{2}}|x-y|^\alpha \big(|\frac{y}{\varepsilon}|^{1-a} +|\frac{y}{\varepsilon}|^{1+a} \big)^{\frac{2N-\alpha}{2}}}dxdy\\
&\leq C\varepsilon^{(\alpha-2N)} \bigg(\int_{\mathbb{R}^N\setminus B_\delta}\frac{1}{\big(|\frac{x}{\varepsilon}|^{1-a} +|\frac{x}{\varepsilon}|^{1+a}  \big)^N}dx\bigg)^{\frac{2N-\alpha}{N}} \\
&= C\varepsilon^{(\alpha-2N)} \bigg(\int_{\mathbb{R}^N\setminus B_\delta}\frac{\varepsilon^{(1+a)N}}{\big(|x|^{1-a}(\varepsilon^{2a} +|x|^{2a})   \big)^N}dx\bigg)^{\frac{2N-\alpha}{N}} \\
&\leq C\varepsilon^{(2N-\alpha)a} \bigg(\int_{|x|>\delta}\frac{1}{|x|^{(1+a)N}   }dx\bigg)^{\frac{2N-\alpha}{N}} \\
&=O(\varepsilon^{(2N-\alpha)a}).
\end{aligned}$$
Thus, \begin{equation}\label{eq3.7}
\begin{aligned}
\int_\Omega\int_\Omega\frac{|\bar{u}_\varepsilon(x)|^{2_\alpha^*} |\bar{u}_\varepsilon(y)|^{2_\alpha^*}}{|x-y|^\alpha}dxdy &\geq S_{H,\alpha}^{ \frac{2N-\alpha}{N-\alpha+2}}-O(\varepsilon^{(2N-\alpha)a}).
\end{aligned}
\end{equation}
On the other hand,
\begin{equation*}
\begin{aligned}
&\Big|\int_\Omega\int_\Omega\frac{|\bar{u}_\varepsilon(x)|^{2_\alpha^*} |\bar{u}_\varepsilon(y)|^{2_\alpha^*}}{|x-y|^\alpha}dxdy- \int_{\mathbb{R}^N}\int_{\mathbb{R}^N} \frac{|u_{\mu,\varepsilon}(x)|^{2_{\alpha}^*}|u_{\mu,\varepsilon}(y)|^{2_{\alpha}^*} }{|x-y|^{\alpha}}dxdy \Big|\\
&\leq \Big|\int_{\mathbb{R}^N}\int_{\mathbb{R}^N}\frac{\big(|\bar{u}_\varepsilon(x)|^{2_\alpha^*} -|u_{\mu,\varepsilon}(x)|^{2_{\mu}^*} \big) |\bar{u}_\varepsilon(y)|^{2_\alpha^*}}{|x-y|^\alpha}dxdy\Big| +\Big| \int_{\mathbb{R}^N}\int_{\mathbb{R}^N} \frac{|u_{\mu,\varepsilon}(x)|^{2_{\alpha}^*}\big(|\bar{u}_\varepsilon(y)|^{2_\alpha^*} -|u_{\mu,\varepsilon}(y)|^{2_{\alpha}^*} \big)}{|x-y|^{\alpha}}dxdy   \Big|\\
&\leq \int_{\mathbb{R}^N}\int_{\mathbb{R}^N}\frac{\big||\eta(x)|^{2_\alpha^*} - 1\big| |u_{\mu,\varepsilon}(x)|^{2_{\alpha}^*} |\eta(y)|^{2_\alpha^*}  |u_{\mu,\varepsilon}(y)|^{2_{\alpha}^*}}{|x-y|^\alpha}dxdy + \int_{\mathbb{R}^N}\int_{\mathbb{R}^N} \frac{|u_{\mu,\varepsilon}(x)|^{2_{\alpha}^*}\big||\eta(y)|^{2_\alpha^*} - 1\big| |u_{\mu,\varepsilon}(y)|^{2_{\alpha}^*}}{|x-y|^{\alpha}}dxdy   \\
&\leq \int_{\mathbb{R}^N\setminus B_\delta}\int_{\mathbb{R}^N}\frac{ |u_{\mu,\varepsilon}(x)|^{2_{\alpha}^*}  |u_{\mu,\varepsilon}(y)|^{2_{\alpha}^*}}{|x-y|^\alpha}dxdy + \int_{\mathbb{R}^N}\int_{\mathbb{R}^N\setminus B_\delta} \frac{|u_{\mu,\varepsilon}(x)|^{2_{\alpha}^*} |u_{\mu,\varepsilon}(y)|^{2_{\alpha}^*}}{|x-y|^{\alpha}}dxdy  \\
&= 2\int_{\mathbb{R}^N\setminus B_\delta}\int_{\mathbb{R}^N}\frac{ |u_{\mu,\varepsilon}(x)|^{2_{\alpha}^*}  |u_{\mu,\varepsilon}(y)|^{2_{\alpha}^*}}{|x-y|^\alpha}dxdy\\
&\leq C\varepsilon^{(\alpha-2N)} \bigg(\int_{\mathbb{R}^N\setminus B_\delta}\frac{1}{\big(|\frac{x}{\varepsilon}|^{1-a} +|\frac{x}{\varepsilon}|^{1+a}  \big)^N}dx\bigg)^{\frac{2N-\alpha}{2N}} \bigg(\int_{\mathbb{R}^N}\frac{1}{\big(|\frac{y}{\varepsilon}|^{1-a} +|\frac{y}{\varepsilon}|^{1+a}  \big)^N}dy\bigg)^{\frac{2N-\alpha}{2N}}\\
&\leq C\varepsilon^{\frac{(2N-\alpha)a}{2}} \bigg(\int_{\mathbb{R}^N\setminus B_\delta}\frac{1}{\big(|x|^{1-a}(\varepsilon^{2a} +|x|^{2a})   \big)^N}dx\bigg)^{\frac{2N-\alpha}{2N}} \bigg(\int_{\mathbb{R}^N}\frac{1}{\big(|y|^{1-a} +|y|^{1+a}  \big)^N}dy\bigg)^{\frac{2N-\alpha}{2N}}\\
&=O(\varepsilon^{\frac{(2N-\alpha)a}{2}}),
\end{aligned}
\end{equation*}
which means that
\begin{equation}\label{eq3.8}
\begin{aligned}
\int_\Omega\int_\Omega\frac{|\bar{u}_\varepsilon(x)|^{2_\alpha^*} |\bar{u}_\varepsilon(y)|^{2_\alpha^*}}{|x-y|^\alpha}dxdy &\leq \int_{\mathbb{R}^N}\int_{\mathbb{R}^N} \frac{|u_{\mu,\varepsilon}(x)|^{2_{\alpha}^*}|u_{\mu,\varepsilon}(y)|^{2_{\alpha}^*} }{|x-y|^{\alpha}}dxdy+O(\varepsilon^{\frac{(2N-\alpha)a}{2}})\\
&=S_{H,\alpha}^{ \frac{2N-\alpha}{N-\alpha+2}}+O(\varepsilon^{\frac{(2N-\alpha)a}{2}}).
\end{aligned}
\end{equation}
By \eqref{eq3.7} and  \eqref{eq3.8}, we get
\begin{equation}\label{eq3.9}
\begin{aligned}
S_{H,\alpha}^{ \frac{2N-\alpha}{N-\alpha+2}}-O(\varepsilon^{(2N-\alpha)a}) \leq \int_\Omega\int_\Omega\frac{|\bar{u}_\varepsilon(x)|^{2_\alpha^*} |\bar{u}_\varepsilon(y)|^{2_\alpha^*}}{|x-y|^\alpha}dxdy
&\leq S_{H,\alpha}^{ \frac{2N-\alpha}{N-\alpha+2}}+O(\varepsilon^{\frac{(2N-\alpha)a}{2}}).
\end{aligned}
\end{equation}
Moreover, by a straightforward computation it follows
\begin{equation}\label{eq3.10}
\begin{aligned}
\int_\Omega|\bar{u}_\varepsilon|^{q+1}dx&=\int_{|x|\leq 2\delta}|\eta ( x) u_{\mu,\varepsilon} ( x)|^{q+1}dx\\
& \geq \int_{|x|\leq \delta}|u_{\mu,\varepsilon} ( x)|^{q+1}dx\\
& \geq C\varepsilon^{\frac{2-N}{2}(q+1)}  \int_{|x|\leq \delta} \frac{1}{\Big( \big|\frac{x}{\varepsilon} \big|^{1-a} +\big|\frac{x}{\varepsilon} \big|^{1+a} \Big)^{\frac{N-2}{2}(q+1) }}dx\\
& = C\varepsilon^{N-\frac{N-2}{2}(q+1)}  \int_{|x|\leq \frac{\delta}{\varepsilon}} \frac{1}{\Big( |x |^{1-a} + |x  |^{1+a} \Big)^{\frac{N-2}{2}(q+1) }}dx\\
& \geq C\varepsilon^{N-\frac{N-2}{2}(q+1)}  \int_{1\leq|x|\leq \frac{\delta}{\varepsilon}} \frac{1}{\Big( |x |^{1-a} + |x  |^{1+a} \Big)^{\frac{N-2}{2}(q+1) }}dx\\
& \geq C\varepsilon^{N-\frac{N-2}{2}(q+1)}  \int_{1\leq|x|\leq \frac{\delta}{\varepsilon}} \frac{1}{ \big|x  \big|^{(1+a)\frac{N-2}{2}(q+1)} }dx\\
& = C\varepsilon^{N-\frac{N-2}{2}(q+1)}  \int_{1}^{\frac{\delta}{\varepsilon}} r^{N-(1+a)\frac{N-2}{2}(q+1)-1}dr\\
&=\begin{cases}
&C\varepsilon^{N-\frac{N-2}{2}(q+1)},~~~~~~~~\text{if }N< (1+a)\frac{N-2}{2}(q+1),\\&C\varepsilon^{N-\frac{N-2}{2}(q+1)}|ln\varepsilon |,~~\text{if }N= (1+a)\frac{N-2}{2}(q+1), \\&C\varepsilon^{\frac{N-2}{2}(q+1)a},~~~~~~~~~~\text{if } N> (1+a)\frac{N-2}{2}(q+1).
\end{cases}\\
\end{aligned}
\end{equation}
In particular, if $q=1$, we have
\begin{equation}\label{eq3.11}
\begin{aligned}
\int_\Omega|\bar{u}_\varepsilon|^{2}dx& \geq\begin{cases}
&C\varepsilon^{2},~~~~~~~~\text{if }N< (N-2)(1+a),\\
&C\varepsilon^{2}|ln\varepsilon |,~~\text{if }N= (N-2)(1+a), \\
&C\varepsilon^{(N-2)a},~~\text{if } N> (N-2)(1+a).
\end{cases}\\
\end{aligned}
\end{equation}
If $\mu<\bar{\mu}-1$, then $N< (N-2)(1+a)$ and $(N-2)a> 2$, thus by \eqref{eq3.6}, \eqref{eq3.9} and  \eqref{eq3.11}, we obtain
$$\begin{aligned}
0<\max_{t\geq0}\mathcal{I}(t\bar{u}_\varepsilon)&\leq \Big(\frac{1}{2}-\frac{1}{2\cdot{2_{\alpha}^*}} \Big) \Bigg( \frac{\int_\Omega|\nabla \bar{u}_\varepsilon|^2dx-\mu  \int_{\Omega} \frac{|\bar{u}_\varepsilon|^2}{|x|^2} dx- \lambda\int_\Omega |\bar{u}_\varepsilon|^2dx }{\Big(\int_\Omega\int_\Omega\frac{|\bar{u}_\varepsilon(x)|^{2_\alpha^*} |\bar{u}_\varepsilon|^{2_\alpha^*}}{|x-y|^\alpha}dxdy \Big)^{\frac{1}{2_\alpha^*}}}\Bigg)^{\frac{2_\alpha^*}{2_\alpha^*-1}}\\
& \leq \frac{N-\alpha+2}{2(2N-\alpha)} \Bigg(\frac{S_{H,\mu}^{\frac{2N-\alpha}{N-\alpha+2}}+  O(\varepsilon^{(N-2)a})- \lambda O(\varepsilon^2) }
{\Big(S_{H,\mu}^{\frac{2N-\alpha}{N-\alpha+2}}- O(\varepsilon^\frac{(N-2)a}{2})\Big)^\frac{N-2}{2N-\alpha}}  \Bigg)^{\frac{2_\alpha^*}{2_\alpha^*-1}}\\
& \leq \frac{N-\alpha+2}{2(2N-\alpha)} \Bigg(S_{H,\mu}+ O(\varepsilon^{(N-2)a})- \lambda O(\varepsilon^2) \Bigg)^{\frac{2N-\alpha}{N+2-\alpha }}\\
&<\frac{N-\alpha+2}{2(2N-\alpha)}S_{H,\mu}^{\frac{2N-\alpha}{N+2-\alpha}}.
\end{aligned}$$
If $\mu=\bar{\mu}-1$, then $N= (N-2)(1+a)$ and $(N-2)a= 2$,  by \eqref{eq3.6}, \eqref{eq3.9} and  \eqref{eq3.11}, we get
$$\begin{aligned}
0<\max_{t\geq0}\mathcal{I}(t\bar{u}_\varepsilon)&\leq \Big(\frac{1}{2}-\frac{1}{2\cdot{2_{\alpha}^*}} \Big) \Bigg( \frac{\int_\Omega|\nabla \bar{u}_\varepsilon|^2dx-\mu  \int_{\Omega} \frac{|\bar{u}_\varepsilon|^2}{|x|^2} dx- \lambda\int_\Omega |\bar{u}_\varepsilon|^2dx }{\Big(\int_\Omega\int_\Omega\frac{|\bar{u}_\varepsilon(x)|^{2_\alpha^*} |\bar{u}_\varepsilon|^{2_\alpha^*}}{|x-y|^\alpha}dxdy \Big)^{\frac{1}{2_\alpha^*}}}\Bigg)^{\frac{2_\alpha^*}{2_\alpha^*-1}}\\
& \leq \frac{N-\alpha+2}{2(2N-\alpha)} \Bigg(\frac{S_{H,\mu}^{\frac{2N-\alpha}{N-\alpha+2}}+  O(\varepsilon^{2})- \lambda C\varepsilon^2|ln \varepsilon|  }
{\Big(S_{H,\mu}^{\frac{2N-\alpha}{N-\alpha+2}}- O(\varepsilon)\Big)^\frac{N-2}{2N-\alpha}}  \Bigg)^{\frac{2_\alpha^*}{2_\alpha^*-1}}\\
& \leq \frac{N-\alpha+2}{2(2N-\alpha)} \Bigg(S_{H,\mu}+ O(\varepsilon^{2})- \frac{\lambda \varepsilon^2|ln \varepsilon|}{\Big(S_{H,\mu}^{\frac{2N-\alpha}{N-\alpha+2}}- O(\varepsilon)\Big)^\frac{N-2}{2N-\alpha}} \Bigg)^{\frac{2N-\alpha}{N+2-\alpha}}\\
&<\frac{N-\alpha+2}{2(2N-\alpha)}S_{H,\mu}^{\frac{2N-\alpha}{N+2-\alpha}}.
\end{aligned}$$
Therefore, $$\max_{t\geq0}\mathcal{I}(t\bar{u}_\varepsilon)< \frac{N-\alpha+2}{2(2N-\alpha)}S_{H, \mu}^{\frac{2N-\alpha}{N+2-\alpha}}.$$
\end{proof}

\textbf{Proof of Theorem \ref{Thm1.2}:}
From Lemma \ref{Lem3.1} and the mountain pass theorem (see \cite{Willem1996}),
there exists a $(PS)_c$ sequence $\{u_n\}$ such that
\begin{equation*}
\mathcal{I}(u_n)\to c~~\text{ and}~~~\mathcal{I}'(u_n)\to 0~\text{in }~ (H_0^1(\Omega))^{-1}
\end{equation*}
at the minimax level
$$c=\inf_{\gamma\in \Gamma}\sup_{t\in[0,1]}\mathcal{I}(\gamma(t)),$$
where  $\Gamma$ is defined by
$$\Gamma:=\Big\{\gamma\in C([0,1], H_0^1(\Omega)): \gamma(0)=0~~\text{and}~~\mathcal{I}(\gamma(1))<0  \Big\}.$$
By the definition of $c$ and Lemma \ref{Lem3.4}, we obtain $c<\frac{N+2-\alpha}{2(2N-\alpha)}S_{H,\mu}^{\frac{2N-\alpha}{N+2-\alpha}} $. Together with Lemma \ref{Lem3.2}, and Lemma \ref{Lem3.3}, we have that the sequence $\{u_n\}$ contains a strongly convergent subsequence and equation \eqref{eq3.1} has a nontrivial solution, observing that $\mathcal{I}(u_0)=c>0$.
\raisebox{-0.5mm}{\rule{2.5mm}{3mm}}\vspace{6pt}

\section{Superlinear perturbation}
In this section, we consider the existence  of solutions for a different perturbation term, that is
\begin{equation}\label{eq4.1}
\left\{\begin{array}{ll}
-\Delta u(x) -\mu\frac{u}{|x|^2}=
\Big( \displaystyle{\int_{\Omega}} \frac{|u(y)|^{2_{\alpha}^*} }{|x-y|^{\alpha}}dy\Big) |u|^{{2_{\alpha}^*}-1}+\lambda u^{q},     \,  &x\in  \Omega, \\
u(x) =0,    \, &  x\in \partial \Omega,
\end{array} \right.
\end{equation}
where $1<q<2^*-1$.
The energy functional $\mathcal{I}$ associated to \eqref{eq4.1} is defined as follows:
\begin{equation*}
\mathcal{I}(u) := \frac{1}{2} \int_{\Omega} |\nabla u|^2dx -\frac{\mu }{2} \int_{\Omega} \frac{|u|^2}{|x|^2} dx-  \frac{\lambda}{q+1} \int_{\Omega} | u|^{q+1}dx -\frac{1}{2\cdot{2_{\alpha}^*}}\int_{\Omega}\int_{\Omega} \frac{|u(x)|^{2_{\alpha}^*}|u(y)|^{2_{\alpha}^*} }{|x-y|^{\alpha}}dxdy.
\end{equation*}

Following the argument for the linear case, with minor modifications we can check that  $\mathcal{I}$ satisfies the Mountain Pass geometry and the $(PS)$ condition at any level $c$, provided $ c< c^*:= \frac{N-\alpha+2}{2(2N-\alpha)}S_{H,\mu}^{\frac{2N-\alpha}{N+2-\alpha}}$. We omit the details to avoid repetitions. So it suffices to find a path
with energy below critical level $c^*$ and we have the following.

\begin{Lem}\label{Lem4.1}
Let $0<\mu < \bar{\mu}$, then  there exists $\varepsilon>0$ small enough such that
$$\max_{t\geq0}\mathcal{I}(t\bar{u}_\varepsilon)\leq c< \frac{N+2-\alpha}{2(2N-\alpha)}S_{H,\mu}^{\frac{2N-\alpha}{N+2-\alpha}},$$
provided that
\begin{itemize}
  \item $N>\frac{2(2a+q+1)}{2a+q-1}$ and  $\lambda>0$ or
  \item $N<\frac{2(2a+q+1)}{2a+q-1}$ and  $\lambda>\lambda_0$, for $\lambda_0$ large enough.
\end{itemize}
\end{Lem}
\begin{proof}
Case I: $N>\frac{2(2a+q+1)}{2a+q-1}$.  In this case, it follows from  $1<q<2^*-1$ that we have $(N-2)a>N-\frac{N-2}{2}(q+1)$.
For any $t>0$,  we have
\begin{equation*}
\begin{aligned}
\mathcal{I}(t\bar{u}_\varepsilon)&= \frac{t^2}{2} \int_{\Omega} |\nabla \bar{u}_\varepsilon|^2dx -\frac{\mu t^2}{2} \int_{\Omega} \frac{|\bar{u}_\varepsilon|^2}{|x|^2} dx-  \frac{\lambda t^{q+1}}{q+1} \int_{\Omega}  |\bar{u}_\varepsilon|^{q+1}dx -\frac{t^{2\cdot{2_{\alpha}^*}}}{2\cdot{2_{\alpha}^*}}\int_{\Omega}\int_{\Omega} \frac{|\bar{u}_\varepsilon(x)|^{2_{\alpha}^*}|\bar{u}_\varepsilon(y)|^{2_{\alpha}^*} }{|x-y|^{\alpha}}dxdy\\
&\leq\frac{t^2}{2}\Big(S_{H,\mu}^{\frac{2N-\alpha}{N-\alpha+2}}+O(\varepsilon^{(N-2)a})\Big) -\lambda t^{q+1} O(\varepsilon^{N-\frac{N-2}{2}(q+1)}) -\frac{t^{2\cdot{2_{\alpha}^*}}}{2\cdot{2_{\alpha}^*}}\Big( S_{H,\mu}^{\frac{2N-\alpha}{N-\alpha+2}}- O(\varepsilon^\frac{(N-2)a}{2} \Big)\\
&=:h(t).
\end{aligned}
\end{equation*}
Since $1 < q <2^*$, then it is easy to see that $h(0)=0$ and $\lim\limits_{t\to+\infty }h(t)=-\infty$. Thus there exists $t_\varepsilon>0$ such that $h(t_\varepsilon)=\max\limits_{t>0}h(t)$ and
\begin{equation*}
\begin{aligned}
t_\varepsilon\Big(S_{H,\mu}^{\frac{2N-\alpha}{N-\alpha+2}}+O(\varepsilon^{(N-2)a})\Big) -\lambda t_\varepsilon^{q} O(\varepsilon^{N-\frac{N-2}{2}(q+1)}) -t_\varepsilon^{2\cdot{2_{\alpha}^*}-1}\Big( S_{H,\mu}^{\frac{2N-\alpha}{N-\alpha+2}}- O(\varepsilon^\frac{(N-2)a}{2} \Big)=0.
\end{aligned}
\end{equation*}
For $\varepsilon>$ sufficiently small,
\begin{equation*}
\begin{aligned}
t_\varepsilon&=\Bigg(\frac{S_{H,\mu}^{\frac{2N-\alpha}{N-\alpha+2}}+O(\varepsilon^{(N-2)a}) -\lambda t_\varepsilon^{q-1} O(\varepsilon^{N-\frac{N-2}{2}(q+1)})}{S_{H,\mu}^{\frac{2N-\alpha}{N-\alpha+2}}- O(\varepsilon^\frac{(N-2)a}{2})} \Bigg)^{\frac{1}{2\cdot{2_{\alpha}^*}-2}}\\
&\leq\Bigg(\frac{S_{H,\mu}^{\frac{2N-\alpha}{N-\alpha+2}}+O(\varepsilon^{(N-2)a}) }{S_{H,\mu}^{\frac{2N-\alpha}{N-\alpha+2}}- O(\varepsilon^\frac{(N-2)a}{2})} \Bigg)^{\frac{1}{2\cdot{2_{\alpha}^*}-2}}\\
&\leq \Big(1+O(\varepsilon^{(N-2)a})   \Big)^{\frac{1}{2\cdot{2_{\alpha}^*}-2}}.
\end{aligned}
\end{equation*}
On the other hand, there exists $t_0 > 0$ independent of $\varepsilon$ such that for $\varepsilon>$  small enough,
\begin{equation*}
\begin{aligned}
t_\varepsilon&=\Bigg(\frac{S_{H,\mu}^{\frac{2N-\alpha}{N-\alpha+2}}+O(\varepsilon^{(N-2)a}) -\lambda t_\varepsilon^{q-1} O(\varepsilon^{N-\frac{N-2}{2}(q+1)})}{S_{H,\mu}^{\frac{2N-\alpha}{N-\alpha+2}}- O(\varepsilon^\frac{(N-2)a}{2})} \Bigg)^{\frac{1}{2\cdot{2_{\alpha}^*}-2}}\\
&\geq \Bigg(\frac{S_{H,\mu}^{\frac{2N-\alpha}{N-\alpha+2}}+O(\varepsilon^{(N-2)a}) -O(\varepsilon^{N-\frac{N-2}{2}(q+1)}) }{S_{H,\mu}^{\frac{2N-\alpha}{N-\alpha+2}}- O(\varepsilon^\frac{(N-2)a}{2})} \Bigg)^{\frac{1}{2\cdot{2_{\alpha}^*}-2}}\\
&>\Big(1  -O(\varepsilon^{N-\frac{N-2}{2}(q+1)})  \Big)^{\frac{1}{2\cdot{2_{\alpha}^*}-2}}\\
&:=t_0>0.
\end{aligned}
\end{equation*}
Thus we have a lower and a upper bound for $t_\varepsilon$, independent of $\varepsilon$. We can argue as in the proof of Lemma \ref{Lem3.4} to conclude that
$$\begin{aligned}
0<\max_{t\geq0}\mathcal{I}(t\bar{u}_\varepsilon)&\leq \Big(\frac{1}{2}-\frac{1}{2\cdot{2_{\alpha}^*}} \Big) \Bigg( \frac{\int_\Omega|\nabla \bar{u}_\varepsilon|^2dx-\mu  \int_{\Omega} \frac{|\bar{u}_\varepsilon|^2}{|x|^2} dx}
{\Big(\int_\Omega\int_\Omega\frac{|\bar{u}_\varepsilon(x)|^{2_\alpha^*} |\bar{u}_\varepsilon|^{2_\alpha^*}}{|x-y|^\alpha}dxdy \Big)^{\frac{1}{2_\alpha^*}}}\Bigg)^{\frac{2_\alpha^*}{2_\alpha^*-1}}-\lambda C \varepsilon^{N-\frac{N-2}{2}(q+1)}\\
& \leq \frac{N-\alpha+2}{2(2N-\alpha)} \Bigg(\frac{S_{H,\mu}^{\frac{2N-\alpha}{N-\alpha+2}}+  O(\varepsilon^{(N-2)a})}
{\Big(S_{H,\mu}^{\frac{2N-\alpha}{N-\alpha+2}}- O(\varepsilon^\frac{(N-2)a}{2})\Big)^\frac{N-2}{2N-\alpha}}  \Bigg)^{\frac{2_\alpha^*}{2_\alpha^*-1}} -\lambda C \varepsilon^{N-\frac{N-2}{2}(q+1)}\\
& \leq \frac{N-\alpha+2}{2(2N-\alpha)} \Bigg(S_{H,\mu}+ O(\varepsilon^{(N-2)a})\Bigg)^{\frac{2N-\alpha}{N+\alpha-2}} -\lambda C \varepsilon^{N-\frac{N-2}{2}(q+1)}\\
& \leq \frac{N-\alpha+2}{2(2N-\alpha)}S_{H,\mu}^{\frac{2N-\alpha}{N+2-\alpha}} + O(\varepsilon^{(N-2)a})-\lambda C \varepsilon^{N-\frac{N-2}{2}(q+1)}\\
&<\frac{N-\alpha+2}{2(2N-\alpha)}S_{H,\mu}^{\frac{2N-\alpha}{N+2-\alpha}}.
\end{aligned}$$

Case II: $N\leq\frac{2(2a+q+1)}{2a+q-1}$.  In this case,  we have $N\geq \frac{N-2}{2}(1+a)(q+1)$ and $(N-2)a\leq N-\frac{N-2}{2}(q+1)$. Similar to Case I, for any fixed $\varepsilon$,
we can obtain that $\max\limits_{t>0} \mathcal{I}(t\bar{u}_\varepsilon)$ can be attained at some $t_{\varepsilon,\lambda}$ with
\begin{equation}\label{eq4.2}
\begin{aligned}
||\bar{u}_\varepsilon||_{\mu}=\lambda t^{q-1}_{\varepsilon,\lambda} \int_{\Omega}  |\bar{u}_\varepsilon|^{q+1}dx -t_{\varepsilon,\lambda}^{2\cdot{2_{\alpha}^*}-2}\int_{\Omega}\int_{\Omega} \frac{|\bar{u}_\varepsilon(x)|^{2_{\alpha}^*}|\bar{u}_\varepsilon(y)|^{2_{\alpha}^*} }{|x-y|^{\alpha}}dxdy.
\end{aligned}
\end{equation}
It follows from \eqref{eq4.2} that $t_{\varepsilon,\lambda}\to 0$  as $\lambda\to +\infty$. Thus
$$\max\limits_{t>0} \mathcal{I}(t\bar{u}_\varepsilon) =\mathcal{I}(t_{\varepsilon,\lambda}\bar{u}_\varepsilon)\leq  \frac{t_{\varepsilon,\lambda}^2}{2}||\bar{u}_\varepsilon||_{\mu} -\frac{t_{\varepsilon,\lambda}^{2\cdot{2_{\alpha}^*}}}{2\cdot{2_{\alpha}^*}}\int_{\Omega}\int_{\Omega} \frac{|\bar{u}_\varepsilon(x)|^{2_{\alpha}^*}|\bar{u}_\varepsilon(y)|^{2_{\alpha}^*} }{|x-y|^{\alpha}}dxdy\to 0~\text{as}~\lambda\to +\infty, $$
which yields that there exists $\lambda_0>0$ such that
$$\max\limits_{t>0} \mathcal{I}(t\bar{u}_\varepsilon) <\frac{ N+2-\alpha }{2(2N-\alpha)}S_{H,\mu}^{\frac{2N-\alpha}{N+2-\alpha}},~\text{for all}~\lambda>\lambda_0.$$
\end{proof}

\textbf{Proof of Theorem \ref{Thm1.3}:}   In order to finish the proof of Theorem \ref{Thm1.3}, we define
$$\Gamma:=\Big\{\gamma\in C([0,1], H_0^1(\Omega)): \gamma(0)=0~~\text{and}~~\mathcal{I}(\gamma(1))<0  \Big\}$$
and the minimax level
$$c=\inf_{\gamma\in \Gamma}\sup_{t\in[0,1]}\mathcal{I}(\gamma(t)).$$
Using Lemma \ref{Lem4.1}, we get that $c<\frac{N+2-\alpha}{2(2N-\alpha)}S_{H,\mu}^{\frac{2N-\alpha}{N+2-\alpha}}$, provided that
\begin{itemize}
  \item $N>\frac{2(2a+q+1)}{2a+q-1}$ and  $\lambda>0$ or
  \item $N<\frac{2(2a+q+1)}{2a+q-1}$ and  $\lambda>\lambda_0$, for $\lambda_0$ large enough.
\end{itemize}
Similarly to the proof of Theorem \ref{Thm1.2}, the equation \eqref{eq4.1} has a nontrivial solution.
\raisebox{-0.5mm}{\rule{2.5mm}{3mm}}\vspace{6pt}

\section{Nonlocal superlinear perturbation}
In this section we finally consider the following nonlocal perturbation,
\begin{equation}\label{eq5.1}
\left\{\begin{array}{ll}
-\Delta u -\mu\frac{u}{|x|^2}=
\Big( \displaystyle{\int_{\Omega}} \frac{|u(y)|^{2_{\alpha}^*} }{|x-y|^{\alpha}}dy\Big) |u|^{{2_{\alpha}^*}-1}+\lambda\Big( \displaystyle{\int_{\Omega}} \frac{|u(y)|^{p} }{|x-y|^{\alpha}}dy\Big) |u|^{p-1} ,     \,  &x\in  \Omega, \\
u(x) =0,    \, &  x\in \partial \Omega.
\end{array} \right.
\end{equation}
Nontrivial solutions of \eqref{eq5.1} are equivalent to nonzero critical points of the following energy functional
\begin{equation*}
\mathcal{J}(u) := \frac{1}{2}\|u\|_\mu^2 -\frac{1}{2\cdot{2_{\alpha}^*}}\int_{\Omega}\int_{\Omega} \frac{|u(x)|^{2_{\alpha}^*}|u(y)|^{2_{\alpha}^*} }{|x-y|^{\alpha}}dxdy -  \frac{\lambda}{2p} \int_{\Omega}\int_{\Omega} \frac{|u(x)|^{p}|u(y)|^{p} }{|x-y|^{\alpha}}dxdy.
\end{equation*}

Similarly as in the previous sections,  one can check that  $\mathcal{J}$ satisfies the Mountain Pass geometry and the $(PS)$ condition at any level $c$ such that $ c< \frac{N+2-\alpha}{2(2N-\alpha)}S_{H,\mu}^{\frac{2N-\alpha}{N+2-\alpha}}$, so we omit the details. Let us just point out that for a bounded sequence $\{u_n\}$ in $H_0^1(\Omega)$,
p to a subsequence, $u_n\rightharpoonup u_0$ in $H_0^1(\Omega)$, and then,
since $u_n\to u_0$ in $L^s(\Omega)$ for $s\in [1, 2^*)$,  we have that
$$ \int_{\Omega}\int_{\Omega} \frac{|u_n(x)|^{p}|u_n(y)|^{p} }{|x-y|^{\alpha}}dxdy \to \int_{\Omega}\int_{\Omega} \frac{|u_0(x)|^{p}|u_0(y)|^{p} }{|x-y|^{\alpha}}dxdy,~~~p\in(1, 2^*_\alpha-1).$$ We therefore turn to the following estimate.

\begin{Lem}\label{Lem5.1}
Let $0<\mu < \bar{\mu}$, then  there exists $\varepsilon>0$ small enough such that
$$\max_{t\geq0}\mathcal{J}(t\bar{u}_\varepsilon)<c< \frac{N+2-\alpha}{2(2N-\alpha)}S_{H,\mu}^{\frac{2N-\alpha}{N+2-\alpha}},$$
provided that
\begin{itemize}
  \item $N>\frac{2a-\alpha+2p}{a+p-2}$ and  $\lambda>0$, or
  \item $N\leq\frac{2a-\alpha+2p}{a+p-2}$ and  $\lambda$ is sufficiently large.
\end{itemize}
\end{Lem}
\begin{proof}
By the definition of $\bar{u}_\varepsilon$, we have
$$\begin{aligned}
\mathbb{D}=&\int_\Omega\int_\Omega\frac{|\bar{u}_\varepsilon(x)|^{p} |\bar{u}_\varepsilon(y)|^{p}}{|x-y|^\alpha}dxdy\\
& \geqslant\int_{B_\delta}\int_{B_\delta}\frac{|\bar{u}_\varepsilon(x) |^{p}|\bar{u}_\varepsilon(y)|^{p}}{|x-y|^\alpha}dxdy \\
&=\int_{B_\delta}\int_{B_\delta} \frac{|u_{\mu,\varepsilon}(x)|^{p}|u_{\mu,\varepsilon}(y) |^{p}}{|x-y|^\alpha}dxdy\\
&\geq\int_{B_\delta}\int_{B_\delta}\frac{C\varepsilon^{(2-N)p}}{ \big(|\frac{x}{\varepsilon}|^{1-a} +|\frac{x}{\varepsilon}|^{1+a} \big)^{\frac{(N-2)p}{2}}|x-y|^\alpha \big(|\frac{y}{\varepsilon}|^{1-a} +|\frac{y}{\varepsilon}|^{1+a} \big)^{\frac{(N-2)p}{2}}}dxdy\\
&= \int_{B_\delta}\int_{B_\delta}\frac{C\varepsilon^{2N-\alpha-(N-2)p}}{ \big(|\frac{x}{\varepsilon}|^{1-a} +|\frac{x}{\varepsilon}|^{1+a} \big)^{\frac{(N-2)p}{2}}|\frac{x}{\varepsilon}-\frac{y}{\varepsilon}|^\alpha \big(|\frac{y}{\varepsilon}|^{1-a} +|\frac{y}{\varepsilon}|^{1+a} \big)^{\frac{(N-2)p}{2}}}d\frac{x}{\varepsilon}d\frac{y}{\varepsilon} \\
&=C\varepsilon^{2N-\alpha-(N-2)p}\int_{B_\delta}\int_{B_\delta}\frac{1}{ \big(|x|^{1-a} +|x|^{1+a} \big)^{\frac{(N-2)p}{2}}|x-y|^\alpha \big(|y|^{1-a} +|y|^{1+a} \big)^{\frac{(N-2)p}{2}}}dxdy\\
&=O(\varepsilon^{2N-\alpha-(N-2)p}).
\end{aligned}$$
It follows from $p<2_\alpha^*$ that $2N-\alpha-(N-2)p>0$.
For any $t>0$,  we have
\begin{equation*}
\begin{aligned}
\mathcal{J}(t\bar{u}_\varepsilon)&= \frac{t^2}{2}|\bar{u}_\varepsilon |^2_\mu -\frac{t^{2\cdot{2_{\alpha}^*}}}{2\cdot{2_{\alpha}^*}}\int_{\Omega}\int_{\Omega} \frac{|\bar{u}_\varepsilon(x)|^{2_{\alpha}^*}|\bar{u}_\varepsilon(y)|^{2_{\alpha}^*} }{|x-y|^{\alpha}}dxdy  -  \frac{\lambda t^{2p}}{2p} \int_{\Omega}\int_{\Omega} \frac{|\bar{u}_\varepsilon(x)|^{p}|\bar{u}_\varepsilon(y)|^{p} }{|x-y|^{\alpha}}dxdy\\
&\leq\frac{t^2}{2}\Big(S_{H,\mu}^{\frac{2N-\alpha}{N-\alpha+2}}+O(\varepsilon^{(N-2)a})\Big) -\frac{t^{2\cdot{2_{\alpha}^*}}}{2\cdot{2_{\alpha}^*}}\Big( S_{H,\mu}^{\frac{2N-\alpha}{N-\alpha+2}}- O(\varepsilon^\frac{(N-2)a}{2} \Big) -\lambda t^{2p} O(\varepsilon^{2N-\alpha-(N-2)p}) \\
&=:\tilde{h}(t).
\end{aligned}
\end{equation*}
Since $1 < p <2_{\alpha}^*$, then it is easy to see that $h(0)=0$ and $\lim\limits_{t\to+\infty }h(t)=-\infty$. Thus there exists $t_\varepsilon>0$ such that $h(t_\varepsilon)=\max\limits_{t>0}f(t)$ and
\begin{equation*}
\begin{aligned}
t_\varepsilon\Big(S_{H,\mu}^{\frac{2N-\alpha}{N-\alpha+2}}+O(\varepsilon^{(N-2)a})\Big) -t^{2\cdot{2_{\alpha}^*}-1}\Big( S_{H,\mu}^{\frac{2N-\alpha}{N-\alpha+2}}- O(\varepsilon^\frac{(N-2)a}{2} \Big)-\lambda t^{2p-1} O(\varepsilon^{2N-\alpha-(N-2)p})  =0.
\end{aligned}
\end{equation*}
For $\varepsilon>$ sufficiently small,
\begin{equation*}
\begin{aligned}
t_\varepsilon&=\Bigg(\frac{S_{H,\mu}^{\frac{2N-\alpha}{N-\alpha+2}}+O(\varepsilon^{(N-2)a}) - \lambda t^{2p-1} O(\varepsilon^{2N-\alpha-(N-2)p})}{S_{H,\mu}^{\frac{2N-\alpha}{N-\alpha+2}}- O(\varepsilon^\frac{(N-2)a}{2})} \Bigg)^{\frac{1}{2\cdot{2_{\alpha}^*}-2}}\\
&\leq\Bigg(\frac{S_{H,\mu}^{\frac{2N-\alpha}{N-\alpha+2}}+O(\varepsilon^{(N-2)a}) }{S_{H,\mu}^{\frac{2N-\alpha}{N-\alpha+2}}- O(\varepsilon^\frac{(N-2)a}{2})} \Bigg)^{\frac{1}{2\cdot{2_{\alpha}^*}-2}}\\
&\leq \Big(1+O(\varepsilon^{(N-2)a})   \Big)^{\frac{1}{2\cdot{2_{\alpha}^*}-2}}:=t_1.
\end{aligned}
\end{equation*}
On the other hand, there exists $t_1 > 0$ independent of $\varepsilon$ such that, for $\varepsilon>$  small enough,
\begin{equation*}
\begin{aligned}
t_\varepsilon
&\geq \Bigg(\frac{S_{H,\mu}^{\frac{2N-\alpha}{N-\alpha+2}}+O(\varepsilon^{(N-2)a}) -O(\varepsilon^{2N-\alpha-(N-2)p }) }{S_{H,\mu}^{\frac{2N-\alpha}{N-\alpha+2}}- O(\varepsilon^\frac{(N-2)a}{2})} \Bigg)^{\frac{1}{2\cdot{2_{\alpha}^*}-2}}\\
&>\Big(1  -O(\varepsilon^{2N-\alpha-(N-2)p })  \Big)^{\frac{1}{2\cdot{2_{\alpha}^*}-2}} :=t_0>0.
\end{aligned}
\end{equation*}
Thus, $t_\varepsilon$ has a lower bound and a upper bound independent of $\varepsilon$.
Now, by the elementary inequality $(m+n)^\varrho<m^\varrho+ \varrho(m+n)^{\varrho-1}n$ for $m,n>0$ and $\varrho\geq1$, we obtain
\begin{equation}\label{eq5.2}
\begin{aligned}
0<\max_{t\geq0}\mathcal{J}(t\bar{u}_\varepsilon)&\leq \Big(\frac{1}{2}-\frac{1}{2\cdot{2_{\alpha}^*}} \Big) \Bigg( \frac{\int_\Omega|\nabla \bar{u}_\varepsilon|^2dx-\mu  \int_{\Omega} \frac{|\bar{u}_\varepsilon|^2}{|x|^2} dx}
{\Big(\int_\Omega\int_\Omega\frac{|\bar{u}_\varepsilon(x)|^{2_\alpha^*} |\bar{u}_\varepsilon|^{2_\alpha^*}}{|x-y|^\alpha}dxdy \Big)^{\frac{1}{2_\alpha^*}}}\Bigg)^{\frac{2_\alpha^*}{2_\alpha^*-1}}-\lambda O( \varepsilon^{2N-\alpha-(N-2)p})\\
& \leq \frac{N-\alpha+2}{2(2N-\alpha)} \Bigg(\frac{S_{H,\mu}^{\frac{2N-\alpha}{N-\alpha+2}}+  O(\varepsilon^{(N-2)a})}
{\Big(S_{H,\mu}^{\frac{2N-\alpha}{N-\alpha+2}}- O(\varepsilon^\frac{(N-2)a}{2})\Big)^\frac{N-2}{2N-\alpha}}  \Bigg)^{\frac{2_\alpha^*}{2_\alpha^*-1}} -\lambda O( \varepsilon^{2N-\alpha-(N-2)p})\\
& =\frac{N-\alpha+2}{2(2N-\alpha)} \Bigg(S_{H,\mu}+ O(\varepsilon^{(N-2)a})\Bigg)^{\frac{2N-\alpha}{N+2-\alpha}} -\lambda O( \varepsilon^{2N-\alpha-(N-2)p})\\
& \leq \frac{N-\alpha+2}{2(2N-\alpha)}S_{H,\mu}^{\frac{2N-\alpha}{N+2-\alpha}} + O(\varepsilon^{(N-2)a})-\lambda O( \varepsilon^{2N-\alpha-(N-2)p}).
\end{aligned}
\end{equation}
Now we distinguish the following cases:
\begin{itemize}
\item[(\romannumeral1)] In the case $N>\frac{2a-\alpha+2p}{a+p-2}$, by \eqref{eq5.2} we get
$$\max_{t\geq0}\mathcal{J}(t\bar{u}_\varepsilon)\leq \frac{N-\alpha+2}{2(2N-\alpha)}S_{H,\mu}^{\frac{2N-\alpha}{N+2-\alpha}} + O(\varepsilon^{(N-2)a})- O( \varepsilon^{2N-\alpha-(N-2)p}). $$
In view of $(N-2)a>2N-\alpha-(N-2)p>0$, we get the conclusion for $\varepsilon$ small enough.
\item[(\romannumeral2)] In the case $N\leq\frac{2a-\alpha+2p}{a+p-2}$, by \eqref{eq5.2} we have
$$\max_{t\geq0}\mathcal{J}(t\bar{u}_\varepsilon)\leq \frac{N-\alpha+2}{2(2N-\alpha)}S_{H,\mu}^{\frac{2N-\alpha}{N+2-\alpha}} + O(\varepsilon^{(N-2)a})-\lambda O( \varepsilon^{2N-\alpha-(N-2)p}). $$
for $\lambda=\varepsilon^{-\theta}$ with $\theta>2N-\alpha-(N-2)p-(N-2)a$, we also get the conclusion.
\end{itemize}
\end{proof}

\textbf{Proof of Theorem \ref{Thm1.4}:} We know that there exists a $(PS)_c$ sequence $\{u_n\}$ of $\mathcal{J}$  with $c< \frac{N+2-\alpha}{2(2N-\alpha)}S_{H,\mu}^{\frac{2N-\alpha}{N+2-\alpha}}$
provided that
\begin{itemize}
  \item $N>\frac{2a-\alpha+2p}{a+p-2}$ and  $\lambda>0$, or
  \item $N\leq\frac{2a-\alpha+2p}{a+p-2}$ and  $\lambda$ is sufficiently large.
\end{itemize}
Similarly as in the previous sections, this yields a nontrivial solution for equation \eqref{eq5.1}.
\raisebox{-0.5mm}{\rule{2.5mm}{3mm}}\vspace{6pt}

\section*{Acknowledgments}
This work was supported by Xingdian talent support program of Yunnan Province, National Natural Science Foundation of China (12261107, 12101546), Yunnan Fundamental Research Projects (202301AU070144, 202301AU070159), Scientific Research Fund of Yunnan Educational Commission (2023J0199, 2023Y0515). A.J. is partially supported by PRIN 2022 "\emph{Pattern formation in nonlinear phenomena}" and is a member of the INDAM Research Group GNAMPA.

\bibliographystyle{plain}
\bibliography{gu}

\begin{thebibliography}{10}

\bibitem{Bahrami2014}
M.~Bahrami, A.~Gro\ss~ardt, S.~Donadi, and A.~Bassi.
\newblock The {S}chr\"{o}dinger-{N}ewton equation and its foundations.
\newblock {\em New J. Phys.}, 16(November):115007, 17, 2014.

\bibitem{Brezis-Nirenberg1983CPAM}
H.~Br\'{e}zis and L.~Nirenberg.
\newblock Positive solutions of nonlinear elliptic equations involving critical
  {S}obolev exponents.
\newblock {\em Comm. Pure Appl. Math.}, 36(4):437--477, 1983.

\bibitem{Cao-Han2004JDE}
D.~Cao and P.~Han.
\newblock Solutions for semilinear elliptic equations with critical exponents
  and {H}ardy potential.
\newblock {\em J. Differential Equations}, 205(2):521--537, 2004.

\bibitem{Cao-Peng2003PAMS}
D.~Cao and S.~Peng.
\newblock A global compactness result for singular elliptic problems involving
  critical {S}obolev exponent.
\newblock {\em Proc. Amer. Math. Soc.}, 131(6):1857--1866, 2003.

\bibitem{Cao-Peng2003JDE}
D.~Cao and S.~Peng.
\newblock A note on the sign-changing solutions to elliptic problems with
  critical {S}obolev and {H}ardy terms.
\newblock {\em J. Differential Equations}, 193(2):424--434, 2003.

\bibitem{Cao-Yan2010CVPDE}
D.~Cao and S.~Yan.
\newblock Infinitely many solutions for an elliptic problem involving critical
  {S}obolev growth and {H}ardy potential.
\newblock {\em Calc. Var. Partial Differential Equations}, 38(3-4):471--501,
  2010.

\bibitem{Catrina-Wang2001}
F.~Catrina and Z.~Wang.
\newblock On the {C}affarelli-{K}ohn-{N}irenberg inequalities: sharp constants,
  existence (and nonexistence), and symmetry of extremal functions.
\newblock {\em Comm. Pure Appl. Math.}, 54(2):229--258, 2001.

\bibitem{Cerami-Solimini-Struwe1986JFA}
G.~Cerami, S.~Solimini, and M.~Struwe.
\newblock Some existence results for superlinear elliptic boundary value
  problems involving critical exponents.
\newblock {\em J. Funct. Anal.}, 69(3):289--306, 1986.

\bibitem{Cerami-Zhong-Zou2015CVPDE}
G.~Cerami, X.~Zhong, and W.~Zou.
\newblock On some nonlinear elliptic {PDE}s with {S}obolev-{H}ardy critical
  exponents and a {L}i-{L}in open problem.
\newblock {\em Calc. Var. Partial Differential Equations}, 54(2):1793--1829,
  2015.

\bibitem{Chen2003JDE}
J.~Chen.
\newblock Existence of solutions for a nonlinear {PDE} with an inverse square
  potential.
\newblock {\em J. Differential Equations}, 195(2):497--519, 2003.

\bibitem{Chen-Zou2012JDE}
Z.~Chen and W.~Zou.
\newblock On an elliptic problem with critical exponent and {H}ardy potential.
\newblock {\em J. Differential Equations}, 252(2):969--987, 2012.

\bibitem{Choquard-Stubbe-Vuffray2008DIE}
P.~Choquard, J.~Stubbe, and M.~Vuffray.
\newblock Stationary solutions of the {S}chr\"{o}dinger-{N}ewton model---an
  {ODE} approach.
\newblock {\em Differential Integral Equations}, 21(7-8):665--679, 2008.

\bibitem{Chou-Chu1993JLMS}
K.~Chou and C.~Chu.
\newblock On the best constant for a weighted {S}obolev-{H}ardy inequality.
\newblock {\em J. London Math. Soc. (2)}, 48(1):137--151, 1993.

\bibitem{Egnell1989IUMJ}
H.~Egnell.
\newblock Elliptic boundary value problems with singular coefficients and
  critical nonlinearities.
\newblock {\em Indiana Univ. Math. J.}, 38(2):235--251, 1989.

\bibitem{Ekeland-Ghoussoub2002BAMS}
I.~Ekeland and N.~Ghoussoub.
\newblock Selected new aspects of the calculus of variations in the large.
\newblock {\em Bull. Amer. Math. Soc. (N.S.)}, 39(2):207--265, 2002.

\bibitem{Felli-Marchini-Terracini2007JFA}
V.~Felli, E.~Marchini, and S~Terracini.
\newblock On {S}chr\"{o}dinger operators with multipolar inverse-square
  potentials.
\newblock {\em J. Funct. Anal.}, 250(2):265--316, 2007.

\bibitem{Ferrero-Gazzola2001JDE}
A.~Ferrero and F.~Gazzola.
\newblock Existence of solutions for singular critical growth semilinear
  elliptic equations.
\newblock {\em J. Differential Equations}, 177(2):494--522, 2001.

\bibitem{Frank-Land-Spector1971RMP}
W.~Frank, D.~Land, and R.~Spector.
\newblock Singular potentials.
\newblock {\em Rev. Modern Phys.}, 43(1):36--98, 1971.

\bibitem{Gao-Yang2017JMAA}
F.~Gao and M.~Yang.
\newblock On nonlocal {C}hoquard equations with {H}ardy-{L}ittlewood-{S}obolev
  critical exponents.
\newblock {\em J. Math. Anal. Appl.}, 448(2):1006--1041, 2017.

\bibitem{Gao-Yang2018SCM}
F.~Gao and M.~Yang.
\newblock The {B}rezis-{N}irenberg type critical problem for the nonlinear
  {C}hoquard equation.
\newblock {\em Sci. China Math.}, 61(7):1219--1242, 2018.

\bibitem{Garcia-Peral1998JDE}
J.~P. Garc\'{\i}a~Azorero and I.~Peral~Alonso.
\newblock Hardy inequalities and some critical elliptic and parabolic problems.
\newblock {\em J. Differential Equations}, 144(2):441--476, 1998.

\bibitem{Ghimenti-Liu-Tang2023}
M.~Ghimenti, M.~Liu, and Z.~Tang.
\newblock Multiple solutions for a fractional {C}hoquard problem with slightly
  subcritical exponents on bounded domains.
\newblock {\em NoDEA Nonlinear Differential Equations Appl.}, 30(2):Paper No.
  28, 27, 2023.

\bibitem{Ghimenti-Pagliardini2019CVPDE}
M.~Ghimenti and D.~Pagliardini.
\newblock Multiple positive solutions for a slightly subcritical {C}hoquard
  problem on bounded domains.
\newblock {\em Calc. Var. Partial Differential Equations}, 58(5):Paper No. 167,
  21, 2019.

\bibitem{Ghoussoub-Yuan2000}
N.~Ghoussoub and C.~Yuan.
\newblock Multiple solutions for quasi-linear {PDE}s involving the critical
  {S}obolev and {H}ardy exponents.
\newblock {\em Trans. Amer. Math. Soc.}, 352(12):5703--5743, 2000.

\bibitem{Guo-Tang2025arXiv}
T.~Guo and Tang X.
\newblock Existence and qualitative properties of solutions for a choquard-type
  equation with hardy potential.
\newblock {\em arXiv: 2312.11855}.

\bibitem{He2022JMAA}
R.~He.
\newblock Infinitely many solutions for the {B}r\'{e}zis-{N}irenberg problem
  with nonlinear {C}hoquard equations.
\newblock {\em J. Math. Anal. Appl.}, 515(2):Paper No. 126426, 24, 2022.

\bibitem{Ho-Perera-Sim2023}
K.~Ho, K.~Perera, and I.~Sim.
\newblock On the {B}rezis-{N}irenberg problem for the {$(p, q)$}-{L}aplacian.
\newblock {\em Ann. Mat. Pura Appl. (4)}, 202(4):1991--2005, 2023.

\bibitem{Jannelli1999JDE}
E.~Jannelli.
\newblock The role played by space dimension in elliptic critical problems.
\newblock {\em J. Differential Equations}, 156(2):407--426, 1999.

\bibitem{Landau-Lifshitz1985Book}
L.~Landau and E.~Lifshitz.
\newblock {\em Quantum mechanics: non-relativistic theory. {C}ourse of
  {T}heoretical {P}hysics, {V}ol. 3}.
\newblock Addison-Wesley Series in Advanced Physics. Pergamon Press, Ltd.,
  London-Paris; Addison-Wesley Publishing Company, Inc., Reading, MA, 1958.
\newblock Translated from the Russian by J. B. Sykes and J. S. Bell.

\bibitem{Li-Yang-Zhou2022SCM}
X.~Li, M.~Yang, and X.~Zhou.
\newblock Qualitative properties and classification of solutions to elliptic
  equations with {S}tein-{W}eiss type convolution part.
\newblock {\em Sci. China Math.}, 65(10):2123--2150, 2022.

\bibitem{Lieb1967SAM}
E.~Lieb.
\newblock Existence and uniqueness of the minimizing solution of {C}hoquard's
  nonlinear equation.
\newblock {\em Studies in Appl. Math.}, 57(2):93--105, 1976/77.

\bibitem{Lieb-Loss2001book}
E.~Lieb and M.~Loss.
\newblock {\em Analysis}, volume~14 of {\em Graduate Studies in Mathematics}.
\newblock American Mathematical Society, Providence, RI, second edition, 2001.

\bibitem{Moroz-Van-Schaftingen2013JFA}
V.~Moroz and J.~Van~Schaftingen.
\newblock Groundstates of nonlinear {C}hoquard equations: existence,
  qualitative properties and decay asymptotics.
\newblock {\em J. Funct. Anal.}, 265(2):153--184, 2013.

\bibitem{Pan-Wen-Yang}
K.~Pan, S.~Wen, and J.~Yang.
\newblock Qualitative analysis to an eigenvalue problem of the hartree type
  brr\'{e}zis-nirenberg problem.
\newblock {\em arXiv:2402.12934 [math.AP].}

\bibitem{Pekar1954}
S.~Pekar.
\newblock Untersuchungen \"{u}ber die elektronentheorie der kristalle.
\newblock {\em Akademie Verlag}, 1954.

\bibitem{Reed-Simon1978Book}
M.~Reed and B.~Simon.
\newblock {\em Methods of modern mathematical physics. {IV}. {A}nalysis of
  operators}.
\newblock Academic Press [Harcourt Brace Jovanovich, Publishers], New
  York-London, 1978.

\bibitem{Ruiz-Willem2003JDE}
D.~Ruiz and M.~Willem.
\newblock Elliptic problems with critical exponents and {H}ardy potentials.
\newblock {\em J. Differential Equations}, 190(2):524--538, 2003.

\bibitem{Tod-Moroz1999N}
P.~Tod and I.~Moroz.
\newblock An analytical approach to the {S}chr\"{o}dinger-{N}ewton equations.
\newblock {\em Nonlinearity}, 12(2):201--216, 1999.

\bibitem{Willem1996}
M.~Willem.
\newblock {\em Minimax theorems}, volume~24 of {\em Progress in Nonlinear
  Differential Equations and their Applications}.
\newblock Birkh\"auser Boston, Inc., Boston, MA, 1996.

\bibitem{Yang-Ye-Zhao2023JDE}
M.~Yang, W.~Ye, and S.~Zhao.
\newblock Existence of concentrating solutions of the {H}artree type
  {B}r\'{e}ezis-{N}irenberg problem.
\newblock {\em J. Differential Equations}, 344:260--324, 2023.

\end{thebibliography}

\end{document}